\documentclass[12pt]{article}

\usepackage{amsmath, amsthm, amssymb, mathtools}
\usepackage{natbib}
\usepackage{hyperref}
\usepackage{geometry}
\usepackage{bm}
\usepackage{enumitem}
\usepackage{booktabs}
\usepackage{array}
\usepackage{xcolor}
\usepackage{tikz}
\usetikzlibrary{arrows.meta, positioning, calc, decorations.markings}

\newtheorem{theorem}{Theorem}[section]
\newtheorem{lemma}[theorem]{Lemma}
\newtheorem{proposition}[theorem]{Proposition}
\newtheorem{corollary}[theorem]{Corollary}
\newtheorem{definition}[theorem]{Definition}
\newtheorem{remark}[theorem]{Remark}
\newtheorem{assumption}[theorem]{Assumption}

\DeclareMathOperator*{\argmax}{arg\,max}

\DeclareMathOperator{\supp}{supp}

\newcommand{\R}{\mathbb{R}}
\newcommand{\E}{\mathbb{E}}
\newcommand{\ind}{\mathbf{1}}

\newcommand{\cQ}{\mathcal{Q}}
\newcommand{\cD}{\mathcal{D}}
\newcommand{\cA}{\mathcal{A}}
\newcommand{\cS}{\mathcal{S}}
\newcommand{\cP}{\mathcal{P}}
\newcommand{\cI}{\mathcal{I}}
\newcommand{\cM}{\mathcal{M}}

\tikzset{
  obs/.style  = {circle, draw, thick, minimum size=1.7em, inner sep=1pt, font=\small},
  lat/.style  = {circle, draw, thick, dashed, minimum size=1.7em, inner sep=1pt,
                 fill=gray!10, font=\small},
  arr/.style  = {-{Stealth[length=5pt,width=4pt]}, thick},
  bidir/.style= {{Stealth[length=5pt,width=4pt]}-{Stealth[length=5pt,width=4pt]},
                 thick, color=blue!70!black, bend left=35},
}

\title{Sharp Causal Bounds for Dynamic Treatment Regimes}
\author{
Alexander Alvarez\\
School of Mathematical and Computational Sciences, \\ University of Prince Edward Island\\
\texttt{alalvarez@upei.ca}
\and
Sebastian.~E.~Ferrando\\
Department of Mathematics, Toronto Metropolitan University\\
\texttt{ferrando@torontomu.ca}
}
\date{July 2026}

\begin{document}
\setlength{\emergencystretch}{2em}
\maketitle
\begin{abstract}
We study dynamic treatment regimes under contemporaneous confounding: at each stage, an unmeasured factor may affect both treatment and the next observed state, but has no further direct effect on later stages. From the observational distribution, the causal graph, and specified structural restrictions on possible next states, we construct at each state--treatment pair the set of transition probabilities compatible with this information. These local sets require no sensitivity parameter and can be combined by backward induction to obtain lower and upper bounds on the expected outcome under any given treatment regime. Our main result shows that these bounds are sharp: their endpoints are exactly the smallest and largest expected outcomes generated by causal models compatible with the same observational distribution, causal graph, and structural restrictions.
The same backward-induction method gives a maximin rule for choosing treatments by maximizing the worst-case expected outcome. Thus, for this class of models with contemporaneous confounding, the framework generalizes the classical g-formula of \citet{Robins1986} for evaluating treatment regimes and the Q-learning framework of \citet{Murphy2003} for selecting them. When the local transition probabilities are identified, the two recursions reduce to these classical methods.
\end{abstract}


\section{Introduction}\label{sec:intro}

A dynamic treatment regime (DTR) $d=(d_0,\ldots,d_{T-1})$ is a sequence
of decision rules, one for each treatment stage, where $d_t$ specifies the
treatment to be assigned from the observed state at time $t$. The quantity
of interest is the expected outcome under the regime,
\[
V^d=\mathbb E[Y^d].
\]
Standard methods for evaluating and selecting treatment regimes, including
g-computation \citep{Robins1986}, inverse-probability weighting
\citep{HernanRobins2020}, and Q-learning \citep{Murphy2003} \footnote{Murphy
\citeyearpar{Murphy2003} does not use the term \emph{Q-learning}, but
develops the equivalent finite-horizon backward recursion in terms of
Q-functions and optimal benefit-to-go functions.}, typically rely
on \emph{sequential exchangeability}. Informally, once the observed history
is taken into account, treatment assignment at the current stage contains
no additional information about the future outcomes that would arise under
the alternative treatment choices. Under this assumption, the transition
probabilities needed to evaluate a regime can be recovered from the
observed conditional distributions, leading to the usual g-formula.
When an unmeasured factor affects both treatment assignment and subsequent
outcomes, sequential exchangeability can fail, and these standard
identification arguments are no longer valid.

\noindent This paper develops a framework for evaluating and selecting DTRs when
sequential exchangeability fails because of \emph{contemporaneous}
confounding: at each stage $t$, an independent latent factor $U_t$ may
affect treatment $A_t$ and the next state $S_{t+1}$, with no persistent
effect on later stages beyond what is carried forward through the observed
state itself. Under this restriction, the possible transition probabilities
at each state--action cell are represented by a set of candidate
distributions. When the transition is identified from the observational
law, this set is a singleton; under contemporaneous confounding it becomes
a Manski-type polytope \citep{Manski1990,Manski1995}; and at an unsupported
cell it is the full simplex.

\noindent These sets are constructed from the observational law $P$, the causal
graph $G$, and the structural restrictions on possible next states.
No sensitivity parameter, instrumental variable, or externally specified
ambiguity radius is introduced. The local sets can then be combined by
backward induction to obtain lower and upper bounds on the expected outcome
under a fixed treatment regime. The calculation uses the fact that, at each
state--action cell, the transition probability can be chosen independently
from its own feasible set. This is the rectangular structure used in robust
dynamic programming \citep{Iyengar2005,NilimElGhaoui2005}.

\noindent For any family of feasible sets with this structure, the backward recursion
computes the exact minimum and maximum expected outcomes over that family.
Our main result goes further for the default sets constructed in \ref{sec:feasible}. Under
the assumed model of contemporaneous confounding, and for the specified
observational law, causal graph, and structural restrictions on possible
next states, these bounds are sharp: they are exactly the smallest and
largest values of $V^d$ that can arise from causal models satisfying those
assumptions.

\noindent The same framework can also be used to choose among treatment regimes.
We adopt a maximin criterion \citep{GilboaSchmeidler1989}: at each stage,
the rule selects the action whose worst-case expected outcome is largest.
This leads to a backward recursion closely related to the one used for
evaluation.

\noindent When every local transition probability is identified, each feasible set is a singleton. In that case, the evaluation procedure reduces to a recursive computation of the classical g-formula, while the selection procedure reduces to the usual Q-function backward recursion underlying Q-learning.

\noindent When some local transition probabilities are not identified, a treatment
regime no longer has a uniquely determined expected outcome. Instead, it
has a range of possible expected outcomes. Our evaluation procedure
computes this range, while the maximin selection rule compares regimes by
their lowest possible expected outcome.

\noindent The contemporaneous-confounding assumption is appropriate when the
unmeasured factor is temporary. Examples include an acute symptom flare or
a same-visit laboratory artifact that affects both the treatment decision
and the next observed state. It is not intended to model persistent
unmeasured characteristics, such as chronic frailty or a fixed
socioeconomic condition, that continue to affect later stages. This restriction is related to the ``memoryless'' or ``fresh'' confounder assumptions used in parts of the confounding-robust reinforcement learning literature; see Remark~\ref{rem:comparison-BrunsSmith2023}. When confounding is stage-specific in this way, the observed process retains the Markov structure needed for dynamic programming. Persistent confounding generally does not.

\paragraph{Related work.} 
A closely related strand of methodological literature is confounding-robust offline reinforcement learning \citep{KallusZhou2020,BrunsSmith2021,BrunsSmith2023,Kausik2024}. This literature commonly models unmeasured confounding through sensitivity restrictions that limit its magnitude. In particular, the stagewise-independent settings of \citet{BrunsSmith2021}, \citet{BrunsSmith2023}, and \citet{Kausik2024} yield ambiguity sets that can be propagated by robust dynamic programming; \citet{KallusZhou2020} instead obtain sharp infinite-horizon policy-value bounds through an occupancy-ratio formulation.


\noindent Our approach makes a different restriction. Rather than limiting the
magnitude of the confounder, we restrict how long its direct effect can
persist. Under contemporaneous confounding, the unmeasured factor acts
within one stage and does not remain as a hidden state in later stages.
This temporal restriction allows us to derive the local sets directly,
without introducing a sensitivity parameter.


\noindent The temporal persistence of confounding also determines whether local uncertainty can be combined across stages. \citet{BrunsSmith2021} shows that stagewise-independent
confounding preserves an observed-state Markov structure that permits
robust dynamic programming, whereas persistent confounding makes policy
evaluation substantially harder. \citet{BrunsSmith2023} further shows
that, under their sensitivity model, the resulting rectangular
construction is exactly sharp with two actions but can be conservative
with more than two. In our setting, Theorem~\ref{thm:sharp} shows that the default
local identified sets combine exactly into the causal identified set for
any finite horizon.

\noindent \citet{Tan2025} develops a complementary sensitivity-analysis approach for longitudinal studies. Their extension of the
marginal sensitivity model allows general longitudinal dependence without
a memorylessness restriction, but introduces sensitivity restrictions at
each period. The paper derives explicit observed-data representations,
via convex optimization, for sharp or conservative population bounds
depending on the sensitivity model.

\noindent Two papers are especially close in subject matter.
\citet{ZhangBareinboim2020} derives graph-based bounds for unidentified
longitudinal transition effects, extending Manski-type natural bounds
\citep{Manski1990} to arbitrary causal diagrams. These bounds are then
used to accelerate online DTR learning. Their results do not, however,
show that the bounds obtained separately at each cell combine into the
sharp identified set for the expected outcome under a fixed regime.
\citet{Saghafian2024} evaluates dynamic treatment regimes over an ambiguity set of causal models
 using an $\alpha$-maximin criterion; our maximin criterion
corresponds to the case $\alpha=1$. Their ambiguity set is specified within a chosen latent-state model class, with bounded unobserved confounding handled through sensitivity parameters. In contrast, our set of possible interventional laws is derived nonparametrically from the observational law and the assumed confounding structure, without a sensitivity parameter, and the resulting bounds are sharp.


\noindent There are also related approaches in biostatistics. One line of work
studies unmeasured confounding through sensitivity analyses that specify
parameters or distributions governing the confounding bias
\citep{RobinsRotnitzkyScharfstein2000,RoseMoodieShortreed2023}. Another line of work uses instrumental variables to guide DTR selection
under partial identification; for instance 
\citet{ChenZhang2021} develop IV-optimal and IV-improved DTRs using a
time-varying instrument and a modified Bellman recursion, while
\citet{Han2024} derives a sharp partial ordering
of counterfactual welfares over regimes using binary instruments. We take
a different route: without instruments or sensitivity parameters, we derive
sharp local feasible sets from the temporal structure imposed on the
confounding and show that they combine into the exact identified set of
interventional laws under each regime, yielding sharp regime-value bounds.

\noindent The response-function construction of \cite{BalkePearl1997}
provides an important precedent for deriving sharp causal bounds from an
observational law and a specified causal model, without introducing an
externally chosen sensitivity parameter, in the instrumental-variable setting
with imperfect compliance. Our local feasible sets and the statewise attaining
construction in Appendix~A use a related constructive sharpness strategy in
the sequential, contemporaneously confounded setting studied here.

\noindent The default construction is also related to the broader graph-based
partial-identification approach of \citet{Duarte2024}. The additional
temporal structure in our model allows the problem to be separated into
small local calculations rather than solved as one global optimization.
The backward induction used to combine these local sets
(Theorem~4.4) is standard robust-MDP methodology
\citep{Iyengar2005,NilimElGhaoui2005}. Our contribution is therefore not
a new dynamic-programming algorithm, but the causal construction of its
local inputs and the proof that, under the assumptions of this paper,
they produce sharp bounds.

\noindent In this paper we treat the observational law $P$ as known. Statistical
estimation of $P$, and inference for the corresponding partially identified regime values when $P$ is estimated from finite data, are outside the scope of this paper.


\paragraph{Organization.}
Section~2 introduces the causal model and derives the post-intervention
Markov factorization. Section~3 defines the local feasible sets and the
default Manski-type construction, and discusses how identification may
differ from one state--action cell to another
\citep{ChenDarwiche2025}. Section~4 develops the backward-induction
results, shows how the g-formula and Q-learning arise as special cases,
and introduces maximin regime selection. Section~5 presents two examples:
a two-period HIV antiretroviral example based on the benchmark of
\citet{Naimi2017}, and a synthetic biomarker example in which the
plug-in and maximin criteria rank the treatment regimes differently.
Section~6 concludes and Appendix \ref{app:thm_sharp} presents some proofs.

\section{Causal Model}\label{sec:model}

\subsection{Variables, Causal Graph, and Dynamic Treatment Regimes}

Fix a horizon $T\geq 1$. For each $t=0,\ldots,T$, let
$\mathcal S_t$ be a finite state space, and let $\mathcal A$ be a
finite treatment space. We consider structural causal models (SCMs) in which
$S_t\in\mathcal S_t$ and $A_t\in\mathcal A$, with structural equations
\begin{equation}\label{eq:scm}
\begin{aligned}
S_0     &= f_0^S(\varepsilon_0^S),\\
U_t     &= f_t^U(\varepsilon_t^U),\\
A_t     &= f_t^A(S_t,U_t,\varepsilon_t^A),\\
S_{t+1} &= f_{t+1}^S(S_t,A_t,U_t,\varepsilon_{t+1}^S),
\end{aligned}
\end{equation}
for $t=0,\ldots,T-1$. 
The state and treatment spaces $\{\mathcal S_t\}_{t=0}^T$ and $\mathcal A$ are fixed throughout, whereas the state spaces of the $U_t$ are not fixed in advance and may depend on the SCM. For each SCM, all exogenous variables appearing in
\eqref{eq:scm} are defined on a common probability space.





\noindent For each $t$, $U_t$ is a contemporaneous variable: it enters
directly only the equations for $A_t$ and $S_{t+1}$. 
The variables $S_t$ and $A_t$ are observed, while $\{U_t\}_{t=0}^{T-1}$ are latent. Thus, the observed trajectory is
\[
(S_0,A_0,S_1,A_1,\ldots,S_{T-1},A_{T-1},S_T).
\]

\noindent
\noindent The contemporaneous restriction is natural when treatment
decisions respond to transient unrecorded information, such as an acute
condition, a short-lived symptom, or other information available at the
current visit, that also affects the next observed state. Any lasting
consequences of this unmeasured factor are then carried forward through
the observed state rather than through a persistent latent variable.

\vspace{.1in}
$\mathcal S_t$ may be
larger than the observational support $\mathrm{supp}(S_t) \subseteq \mathcal{S}_t$, allowing, in principle, 
for states that are possible but not observed. At the terminal time $T$, the final state $S_T\in\mathcal S_T$ is observed
and no further treatment is assigned. The outcome is
\[
Y=y(S_T)\in\mathcal Y\subseteq\mathbb R,
\]
where $y:\mathcal S_T\to\mathcal Y$ is a known function. When
$\mathcal S_T\subseteq\mathbb R$, this includes the case $y=\mathrm{id}$.
The full observable sequence is
$\mathbf{V}=(S_0,A_0,S_1,A_1,\ldots,S_{T-1},A_{T-1},S_T,Y)$, taking values
in the product space
\[
\mathcal V \;:=\; \mathcal S_0\times\cA\times\mathcal S_1\times\cA\times\cdots
\times\mathcal S_{T-1}\times\cA\times\mathcal S_T\times\mathcal Y.
\]
Since $Y$ is a
deterministic function of $S_T$,  its law is determined by the law of the
state--action subsequence $(S_0,A_0,\ldots,A_{T-1},S_T)$, and every
trajectory-law factorization below is stated for that subsequence, with
the point mass on $Y$ left implicit.

\vspace{.2in}
\noindent The state $S_t$ may include the entire observed history up to
time $t$. Thus, although the regime below is written as a function of
$S_t$, it includes history-dependent treatment rules by taking $S_t$ to
contain the relevant past states and treatments.

\begin{definition}[Deterministic DTR]
A \emph{deterministic dynamic treatment regime}, or simply a \emph{regime}, is a sequence
$d=(d_0,\ldots,d_{T-1})$ with each $d_t:\mathcal{S}_t\to\cA$.
\end{definition}

\begin{assumption}[Modularity]\label{ass:modularity}
For any regime $d$, intervening according to $d$ replaces
$A_t=f_t^A(S_t,U_t,\varepsilon_t^A)$ by the deterministic assignment
$A_t=d_t(S_t)$ at every $t$, leaving all other structural equations
in~\eqref{eq:scm} unchanged.  We refer to this property as modularity. 
This is a standing assumption throughout the paper
\end{assumption}

\begin{remark}[Pushforward representation]\label{rem:pushforward}
Fix an SCM $M$ of the form \eqref{eq:scm}, and let
$(\Omega_M, Q_M)$ be the probability space carrying its
exogenous random variables. Recursively solving the structural equations
defines a random trajectory, equivalently a map
\[
\Phi_M:\Omega_M\longrightarrow\mathcal V.
\]
The \emph{observational law} $P_M$ is the probability distribution induced
on $\mathcal V$ by $\Phi_M$; explicitly, for every $B\subseteq\mathcal V$,
\[
P_M(B)=Q_M\bigl(\Phi_M^{-1}(B)\bigr).
\]
Under a
regime $d$, modularity (Assumption~\ref{ass:modularity}) replaces only
the treatment equations,  giving a second map into $\mathcal V$ whose
induced distribution is the \emph{interventional law} $P^d_M$.
Everything below uses $P_M$ and $P^d_M$ as ordinary (finite)
probability distributions on $\mathcal V$.  
\end{remark}
\noindent The causal value of regime $d$ in model $M$ is $V_M^d=\E^{P_M^d}[Y].$

\begin{assumption}[Independent noise]\label{ass:npc}
For each SCM, the exogenous noise terms $\varepsilon_0^S$ and
$\{\varepsilon_t^U,\varepsilon_t^A,\varepsilon_{t+1}^S\}_{t=0}^{T-1}$
in~\eqref{eq:scm} are mutually independent.
\end{assumption}

\noindent Assumption~\ref{ass:npc}, together with the recursive structure
in~\eqref{eq:scm}, implies that $\{U_t\}_{t=0}^{T-1}$ are mutually
independent and that $U_t\perp\!\!\!\perp S_t$ for every $t$. Thus, each
$U_t$ is a distinct contemporaneous factor rather than a persistent hidden
state.

\begin{remark}\label{rem:classical_framework}
Model ~\eqref{eq:scm} is closely related to the classical framework of \citet{Robins1986} and
\citet{HernanRobins2020}, which works with a time-varying covariate $L_t$ and
defines the history available at visit $t$ as $\bar H_t := (L_0, A_0, L_1, A_1,
\dots, A_{t-1}, L_t)$, assuming directly that treatment carries no further
information about the future counterfactual trajectory given this history,
$A_t \perp\!\!\!\perp (S_{t+1}^{\bar a}, \dots, S_T^{\bar a}, Y^{\bar a}) \mid
\bar H_t$ for every $t$ and every compatible treatment history $\bar a$, called
the sequential exchangeability assumption. Taking the generic state to be the
full history, $S_t := \bar H_t$, and choosing $f^A_t$ and $f^S_{t+1}$ so that
neither depends on $U_t$, equation ~\eqref{eq:scm} reduces to
\[
A_t = f^A_t(S_t, \varepsilon^A_t), \qquad
S_{t+1} = f^S_{t+1}(S_t, A_t, \varepsilon^S_{t+1}),
\]
with no latent common cause of $A_t$ and $S_{t+1}$ at any step. The
absence of unobserved confounders $U_t$, together with Assumption ~\ref{ass:npc}'s independent noise, makes this model satisfy the 
sequential exchangeability assumption in that special case. Relative to the classical
framework, model ~\eqref{eq:scm}  then differs in two opposite directions. First, with $U_t$ inert at every step, model (1) is \emph{more restrictive}: sequential
exchangeability is a distributional condition on $\bar H_t$, $A_t$, and the
future trajectory and other
data-generating processes besides independent-noise SCMs can satisfy it; model
~\eqref{eq:scm} fixes one particular such mechanism. 
This added restrictiveness does not affect the usual single-regime
identification under sequential exchangeability. It does, however, impose
additional restrictions on the joint distribution of counterfactual variables
under different regimes, a distinction made precise by
\citet{Richardson2013}.
Second, by allowing $U_t$ to be active, model ~\eqref{eq:scm} is also \emph{more general}: it accommodates a latent common cause of $A_t$ and $S_{t+1}$ at each step, that is, unmeasured confounding of exactly the kind the classical framework's assumptions exclude outright.
\end{remark}

\noindent Fix $G$ to be an acyclic directed mixed graph (ADMG) with
directed skeleton
\[
S_t\to A_t,\quad S_t\to S_{t+1},\quad A_t\to S_{t+1}\quad(t=0,\ldots,T-1),
\qquad S_T\to Y,
\]
and a bidirected edge $A_t\leftrightarrow S_{t+1}$ at each step $t$ where
the edge is \emph{permitted}.  
As in a latent projection
\citep{Richardson2003}, such an edge indicates that $U_t$ is permitted
to affect \emph{both} the $A_t$ and $S_{t+1}$ equations.  Particular
structural functions may nevertheless be insensitive to $U_t$ in one
or both equations.  Absence of the edge excludes a latent common cause
of $A_t$ and $S_{t+1}$.
Definition~\ref{def:admissible_class} below makes this
precise.  Since $Y$
is a known deterministic function of $S_T$ alone and there is no terminal
treatment $A_T$, $S_T\to Y$ is the only edge into $Y$.
Figure~\ref{fig:dag} depicts the $T=2$ case,  Its directed skeleton matches the
canonical two-period HIV antiretroviral example of \citet{Robins2000} and
\citet{Naimi2017}, revisited in Section~\ref{subsec:hiv}.  The bidirected edges 
are a hypothetical extension of that example,  not a feature of the original analysis which assumes
sequential exchangeability.\\

\noindent The graph $G$ is fixed throughout the paper and is supplied by the
analyst rather than inferred from the observational law $P(\mathbf v)$.
At each step $t$, absence of the bidirected edge
$A_t\leftrightarrow S_{t+1}$ is a \emph{structural exclusion}: no
admissible SCM may contain a latent common cause of $A_t$ and
$S_{t+1}$. Presence of the edge permits such contemporaneous
confounding, although particular admissible structural functions may
be insensitive to the corresponding latent factor.

\begin{figure}[t]
\centering
\begin{tikzpicture}[node distance=1.5cm and 1.7cm]
  \node[obs] (S0)  {$S_0$};
  \node[lat] (U0)  [above right=0.7cm and 2.9cm of S0] {$U_0$};
  \node[obs] (A0)  [right=1.7cm of S0] {$A_0$};
  \node[obs] (S1)  [right=1.7cm of A0] {$S_1$};
  \node[lat] (U1)  [above right=0.7cm and 2.9cm of S1] {$U_1$};
  \node[obs] (A1)  [right=1.7cm of S1] {$A_1$};
  \node[obs] (S2)  [right=1.7cm of A1] {$S_2$};
  \node[obs] (Y)   [right=1.4cm of S2] {$Y$};

  \draw[arr] (S0) -- (A0);
  \draw[arr] (S0) to[bend right=18] (S1);
  \draw[arr] (A0) -- (S1);
  \draw[arr] (S1) -- (A1);
  \draw[arr] (S1) to[bend right=18] (S2);
  \draw[arr] (A1) -- (S2);
  \draw[arr] (S2) -- (Y);

  \draw[arr, dashed] (U0) -- (A0);
  \draw[arr, dashed] (U0) -- (S1);
  \draw[arr, dashed] (U1) -- (A1);
  \draw[arr, dashed] (U1) -- (S2);

  \node[font=\small, above=0.05cm of S0, xshift=-1.5cm]
        {(a) Latent-variable SCM, $T=2$};
\end{tikzpicture}

\bigskip

\begin{tikzpicture}[node distance=1.5cm and 1.7cm]
  \node[obs] (S0)  {$S_0$};
  \node[obs] (A0)  [right=1.7cm of S0] {$A_0$};
  \node[obs] (S1)  [right=1.7cm of A0] {$S_1$};
  \node[obs] (A1)  [right=1.7cm of S1] {$A_1$};
  \node[obs] (S2)  [right=1.7cm of A1] {$S_2$};
  \node[obs] (Y)   [right=1.4cm of S2] {$Y$};

  \draw[arr] (S0) -- (A0);
  \draw[arr] (S0) to[bend right=18] (S1);
  \draw[arr] (A0) -- (S1);
  \draw[arr] (S1) -- (A1);
  \draw[arr] (S1) to[bend right=18] (S2);
  \draw[arr] (A1) -- (S2);
  \draw[arr] (S2) -- (Y);

  \draw[bidir] (A0) to node[above, font=\scriptsize, yshift=2pt]
       {$U_0$} (S1);
  \draw[bidir] (A1) to node[above, font=\scriptsize, yshift=2pt]
       {$U_1$} (S2);

  \node[font=\small, above=0.05cm of S0, xshift=-2.0cm]
        {(b) Observed ADMG $G$};
\end{tikzpicture}

\caption{Causal graphs for the $T=2$ case.
  \textbf{(a)}~Full SCM; dashed arrows involve latents $U_0,U_1$.
  \textbf{(b)}~Observed ADMG $G$; blue bidirected edges $A_t\leftrightarrow
  S_{t+1}$ (labelled by their latent source) encode contemporaneous
  confounding. The terminal state $S_2$ determines the known outcome
  $Y=y(S_2)$, drawn as an explicit node with a single edge $S_2\to Y$.}
\label{fig:dag}
\end{figure}
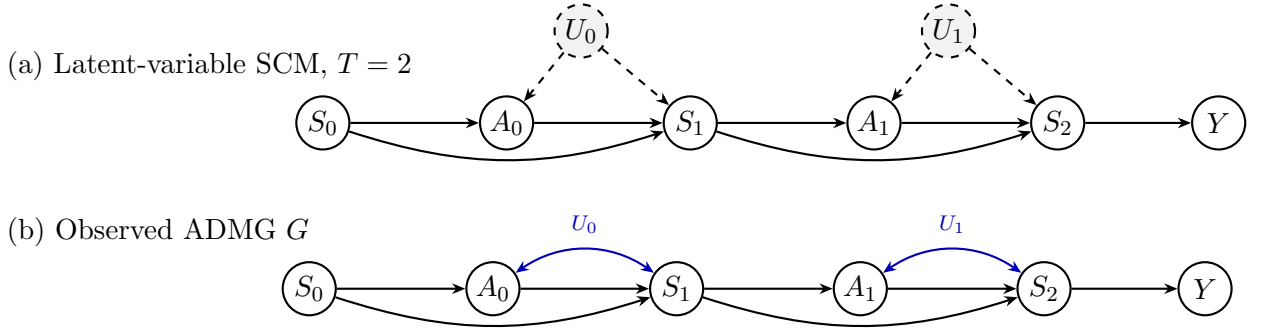

\noindent The edges $A_t\leftrightarrow S_{t+1}$ are the
\emph{only} bidirected edges $G$ can contain; persistent latent variables
inducing longer-range dependence, such as $A_t\leftrightarrow S_{t+k}$ for
$k\geq2$, fall outside the contemporaneous-confounding model studied here.

\begin{definition}[Structural transition kernel]\label{def:kernel}
For a model $M$ of the form~\eqref{eq:scm}, every $t\in\{0,\ldots,T-1\}$,
$s\in\mathcal S_t$, and $a\in\cA$, define the \emph{structural transition
kernel}
\[
\kappa_{M,t}(\cdot\mid s,a) \;:=\;
\mathcal L\bigl(f_{t+1}^S(s,a,U_t,\varepsilon_{t+1}^S)\bigr),
\]
the law, under $M$'s exogenous distribution $Q_M$, of the right-hand side of
the $S_{t+1}$-equation in~\eqref{eq:scm} evaluated at the fixed
hypothetical pair $(s,a)$.
\end{definition}

\noindent 
The random variable whose law defines
$\kappa_{M,t}(\cdot\mid s,a)$ is obtained from the $S_{t+1}$-equation
in~\eqref{eq:scm} by fixing $S_t=s$ and $A_t=a$, while leaving
$U_t$ and $\varepsilon_{t+1}^S$ random under $Q_M$.  Thus
$\kappa_{M,t}(\cdot\mid s,a)$ is part of the specification of $M$ and
is defined at every $(t,s,a)$. Here and below, we use \emph{cell} to mean a local time-state-action triple $(t,s,a)$. Accordingly, the feasible sets below are indexed by
the cell $(t,s,a)$.

\begin{definition}[Local structural successor sets]
\label{def:structural_successors}
For each $t\in\{0,\ldots,T-1\}$, $s\in\mathcal S_t$, and
$a\in\cA$, fix a nonempty set
\[
\mathcal S_{t+1}(s,a)\subseteq \mathcal S_{t+1},
\]
called the \emph{local structural successor set}, or simply \emph{structural successor set} at $(t,s,a)$. 
An SCM $M$ of the form~\eqref{eq:scm} \emph{respects the structural
successor sets} if, for every $t,s,a$,
\[
f_{t+1}^S(s,a,u,e)\in\mathcal S_{t+1}(s,a)
\qquad\text{for all admissible }u,e.
\]
\end{definition}

\noindent For convenience, collect the local structural successor sets as
\[
\mathbf S
:=
\bigl(\mathcal S_{t+1}(s,a)\bigr)_{
t=0,\ldots,T-1;\,
s\in\mathcal S_t;\,
a\in\cA}.
\]
The family $\mathbf S$ is fixed throughout the paper. Its role is simply
to record, for each $(s,a)$, which values of $S_{t+1}$ the model allows.
Every SCM considered below will be required to respect these sets.
For example, in the full-history formulation,
$\mathcal S_{t+1}(s,a)$ contains only histories that validly extend
$(s,a)$.  When no restriction beyond the ambient state space is intended,
one may simply take $\mathcal S_{t+1}(s,a)=\mathcal S_{t+1}$.

\begin{definition}\label{def:admissible_class}
Fix the local structural successor family $\mathbf S$. Let
$\cM(P,G,\mathbf S)$ denote the collection of SCMs $M$ of the
form~\eqref{eq:scm} that satisfy Assumption~\ref{ass:npc}, whose induced
observational law (Remark~\ref{rem:pushforward}) reproduces the fixed
observational law, $P_M=P$, that are \emph{compatible} with the fixed
graph $G$, and that respect the local structural successor sets of
Definition~\ref{def:structural_successors}.
Graph compatibility has the same meaning as above:
if $A_t\leftrightarrow S_{t+1}$ is absent from $G$, no latent variable
may affect both $A_t$ and $S_{t+1}$ at step $t$; if the edge is present,
such a common cause is permitted but not required.
\end{definition}

\noindent Every SCM in $\cM(P,G,\mathbf S)$ shares the same observational law
$P_M=P$, while the induced interventional laws $P^d_M$ and structural
kernels $\kappa_{M,t}$ need not coincide across $\cM(P,G,\mathbf S)$; this
variation is the source of the causal ambiguity studied below.

\begin{remark}\label{rem:true_scm}
We assume throughout that the data are generated by some SCM
$M^\star\in\cM(P,G,\mathbf S)$. This SCM is a single, fixed object, but
in general unknown to the analyst, who knows only that
$M^\star\in\cM(P,G,\mathbf S)$. Since $M^\star\in\cM(P,G,\mathbf S)$,
$P_{M^\star}=P$. Throughout the paper, we write
\[
P^d:=P^d_{M^\star},\qquad
\kappa_t(\cdot\mid s,a):=\kappa_{M^\star,t}(\cdot\mid s,a)
\]
for the corresponding interventional law and structural kernel.
\end{remark}

\begin{assumption}[Finite compatible structural successor specification]
\label{ass:finite}
The state spaces $\mathcal S_t$, $t=0,\ldots,T$, the treatment space
$\cA$, and the outcome space $\mathcal Y$ are finite. Moreover,
\[
\cM(P,G,\mathbf S)\neq\varnothing.
\]
\end{assumption}

\noindent The nonemptiness condition requires the externally specified local
structural successor sets, the graph $G$, and the observational law $P$
to be mutually compatible. In particular, it ensures that the maintained model
class contains at least one SCM satisfying all of the structural,
graphical, and observational restrictions.

\begin{remark}[Observed-support compatibility]
\label{rem:observed_support_compatibility}
If $P(S_t=s,A_t=a)>0$, then
\[
\operatorname{supp} P(S_{t+1}\mid S_t=s,A_t=a)
\subseteq \mathcal S_{t+1}(s,a).
\]
Indeed, by Assumption~\ref{ass:finite}, choose
$M^0\in\cM(P,G,\mathbf S)$. Since $M^0$ respects the local structural
successor sets, on the event $\{S_t=s,A_t=a\}$ its structural equation
for $S_{t+1}$ takes values in $\mathcal S_{t+1}(s,a)$ almost surely.
Because $P_{M^0}=P$, the claimed inclusion follows.
\end{remark}

\begin{remark}
Membership of $z$ in $\mathcal S_{t+1}(s,a)$ means only that the
maintained structural specification permits $z$ as a successor of
$(s,a)$. It does not require $z$ to occur under the true SCM, or under
every SCM in $\cM(P,G,\mathbf S)$.
\end{remark}

\noindent Finiteness (Assumption~\ref{ass:finite}) is imposed for expositional and
computational simplicity; many of the results in this paper have natural extensions to continuous state spaces under suitable regularity conditions; we do not pursue this extension here.

\vspace{.05in}
\noindent Finally,  model~\eqref{eq:scm} assumes that past observed history affects treatment
and transitions only through $S_t$. Taking $S_t=\bar H_t$
(Remark~\ref{rem:classical_framework}) makes this automatic; a
lower-dimensional summary $S_t=\phi(\bar H_t)$ is often used instead for
tractability \citep{Shortreed2011,Lizotte2016}, at the cost of the additional assumption that this summary is sufficient to represent the dependence on past observed history in model~\eqref{eq:scm}. Everything
below is stated for generic $S_t$, so either choice is compatible with the
same framework.

\subsection{The Post-Intervention Factorization}\label{sec:factorization}

Fix a model $M$ of the form~\eqref{eq:scm}. Under regime $d$, the law
$P_M^d$ (Remark~\ref{rem:pushforward}) governs the trajectory obtained
by assigning treatments according to $d$ under $M$. To derive the
structure of $P_M^d$, apply the intervention specified by $d$: replacing
$A_t=f_t^A(S_t,U_t,\varepsilon_t^A)$ by $A_t=d_t(S_t)$
(Assumption~\ref{ass:modularity}) removes every original input to $A_t$,
including $U_t$:
\begin{equation}\label{eq:scm_mutilated}
S_0=f_0^S(\varepsilon_0^S),\quad
U_t=f_t^U(\varepsilon_t^U),\quad
A_t=d_t(S_t),\quad
S_{t+1}=f_{t+1}^S(S_t,A_t,U_t,\varepsilon_{t+1}^S).
\end{equation}

\noindent 
Under intervention, the arrow $U_t\to A_t$ is severed while $U_t\to
S_{t+1}$ remains. Thus $U_t$ has only one observed child, $S_{t+1}$;
any influence on later observed variables is mediated through this observed
node. Consequently, $U_t$ induces no bidirected edge in the latent projection
\citep{Richardson2003}. Since
Assumption~\ref{ass:npc} makes the $\{U_t\}$ mutually independent, the
observed marginal graph $G_{\bar d}^{\,\mathrm{obs}}$ of the mutilated SCM
\eqref{eq:scm_mutilated} acquires no bidirected edges at all (see DAG in Figure~\ref{fig:dag_mutilated}).

\begin{figure}[t]
\centering
\begin{tikzpicture}[node distance=1.5cm and 1.7cm]
  \node[obs] (S0)  {$S_0$};
  \node[obs,fill=gray!15] (A0)  [right=1.7cm of S0] {$A_0$};
  \node[obs] (S1)  [right=1.7cm of A0] {$S_1$};
  \node[obs,fill=gray!15] (A1)  [right=1.7cm of S1] {$A_1$};
  \node[obs] (S2)  [right=1.7cm of A1] {$S_2$};
  \node[obs] (Y)   [right=1.4cm of S2] {$Y$};

  \draw[arr] (S0) -- (A0);
  \draw[arr] (S0) to[bend right=18] (S1);
  \draw[arr] (A0) -- (S1);
  \draw[arr] (S1) -- (A1);
  \draw[arr] (S1) to[bend right=18] (S2);
  \draw[arr] (A1) -- (S2);
  \draw[arr] (S2) -- (Y);
\end{tikzpicture}
\caption{Mutilated ADMG $G^{\,\mathrm{obs}}_{\bar d}$ for the $T=2$ case of
Figure~\ref{fig:dag} (shaded~$=$~determined by $A_t=d_t(S_t)$)): severing $U_t\!\to\!A_t$
eliminates all bidirected edges, yielding an ordinary DAG. The terminal
state $S_2$ determines $Y=y(S_2)$ deterministically.}
\label{fig:dag_mutilated}
\end{figure}
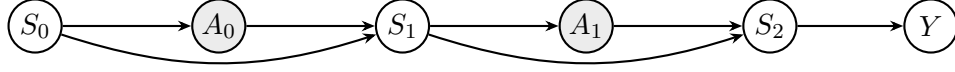

\noindent
By \eqref{eq:scm_mutilated} and Assumption~\ref{ass:npc}, conditional on
$(S_t,A_t)$ the next state $S_{t+1}$ is independent of the earlier
observed history. Hence the post-intervention law is Markov with respect
to $G_{\bar d}^{\,\mathrm{obs}}$.

Reading off the parent sets from
the directed skeleton we get $\mathrm{pa}(S_0)=\emptyset$,
$\mathrm{pa}(A_t)=\{S_t\}$, $\mathrm{pa}(S_{t+1})=\{S_t,A_t\}$ . Then,  the joint
law factorizes as
\begin{equation}\label{eq:dag_factor}
P_M^d(\mathbf{v})
= P_M^d(s_0)\prod_{t=0}^{T-1} P_M^d(a_t\mid s_t)\cdot P_M^d(s_{t+1}\mid s_t,a_t).
\end{equation}
As an ordinary chain-rule factorization, \eqref{eq:dag_factor} is
meaningful as written wherever the conditioning events have positive
$P_M^d$-probability. The theorem below extends \eqref{eq:dag_factor} to a formula valid at
\emph{every} trajectory, including those passing through states or
state--action pairs that are null or unreachable under $P_M^d$.  It
replaces the conditional transition probability by the always-defined
structural kernel $\kappa_{M,t}(\cdot\mid s_t,a_t)$ of
Definition~\ref{def:kernel}, and uses the deterministic treatment rule
$A_t=d_t(S_t)$ explicitly.

\begin{theorem}\label{thm:factorization}
Fix any model $M$ of the form~\eqref{eq:scm} satisfying
Assumption~\ref{ass:npc}, with observational law $P_M$. Under the
standing modularity assumption (Assumption~\ref{ass:modularity}), for
every regime $d$ and every trajectory $\mathbf v$,
\begin{equation}\label{eq:trunc_factor}
P^d_M(\mathbf{v})
\;=\;
P_M(s_0)
\prod_{t=0}^{T-1}
\ind\{a_t=d_t(s_t)\}
\cdot
\kappa_{M,t}(s_{t+1}\mid s_t,a_t).
\end{equation}
In particular, when $M\in\cM(P,G,\mathbf S)$, $P_M(s_0)=P(s_0)$.
\end{theorem}

\begin{proof}
Under Assumption~\ref{ass:modularity}, the intervention replaces the
treatment equation by $A_t=d_t(S_t)$, while the transition equation
remains $S_{t+1}=f_{t+1}^S(S_t,A_t,U_t,\varepsilon_{t+1}^S)$. At a fixed
state--action pair $(s_t,a_t)$, the next-state mechanism is therefore
the structural kernel $\kappa_{M,t}(\cdot\mid s_t,a_t)$ of
Definition~\ref{def:kernel}, regardless of whether $(s_t,a_t)$ has
positive probability under $P_M^d$. Since the fresh exogenous variables
at step $t$ are independent of the preceding trajectory
(Assumption~\ref{ass:npc}), and $A_t=d_t(S_t)$ is deterministic given
$S_t$, multiplying these successive transition mechanisms together with
the deterministic treatment indicators $\ind\{a_t=d_t(s_t)\}$, starting
from $P_M^d(s_0)=P_M(s_0)$, yields the stated factorization
\eqref{eq:trunc_factor} at every trajectory $\mathbf v$.
\end{proof}

\begin{remark}\label{rem:factorization_suppress}
Specializing Theorem~\ref{thm:factorization} to $M=M^\star$, in the
notation already fixed by Remark~\ref{rem:true_scm}, \eqref{eq:trunc_factor}
reads $P^d(\mathbf v)=P(s_0)\prod_{t=0}^{T-1}\ind\{a_t=d_t(s_t)\}\,
\kappa_t(s_{t+1}\mid s_t,a_t)$. Theorem~\ref{thm:factorization} is
invoked in this suppressed form throughout the body of the paper, and in
its explicit model-indexed form in Appendix~\ref{app:thm_sharp}, where
the argument ranges over $M\in\cM(P,G,\mathbf S)$.
\end{remark}

\begin{remark}\label{rem:comparison-BrunsSmith2023}
The memorylessness condition of \citet{BrunsSmith2023} is close to, but
weaker than, Assumption~\ref{ass:npc} in one respect. It allows
the distribution of \(U_t\) to depend on the current observed state \(S_t\),
whereas our assumption implies \(U_t \perp\!\!\!\perp S_t\). Their marginal-MDP result
is closely related to Theorem~\ref{thm:factorization}: in both
settings, the absence of persistent latent confounding implies that the
post-intervention process is Markov in the observed state after the latent
variables are marginalized. In our setting, the additional independence lets
us relate the structural transition kernels directly to the observational law.
Appendix~A then uses this relationship in a statewise construction showing
that the local feasible sets combine into the exact causal identified set for
each regime.
\end{remark}

\noindent Later arguments will also use a sequential factorization of the observational law $P(v)$. Because the observed graph may contain bidirected edges
$A_t\leftrightarrow S_{t+1}$, this factorization does not in general
follow from the graph alone. It does, however, follow from the stage-local structure of model~\eqref{eq:scm},  as the next proposition shows. Throughout, we adopt the convention that an ordinary conditional probability is $0$ whenever its conditioning event has probability zero, so that a product containing such a factor is $0$; this lets the factorization below be stated as holding at \emph{every} trajectory, not only those with positive-probability prefixes.

\begin{proposition}\label{prop:obs_factorization}
Under Assumption~\ref{ass:npc}, for every SCM $M$
of the form~\eqref{eq:scm}, the observational law $P_M$ admits a
factorization of the form
\begin{equation}\label{eq:obs_factor}
P_M(\mathbf v)
\;=\;
P_M(s_0)\prod_{t=0}^{T-1} P_M(a_t,s_{t+1}\mid s_t)
\;=\;
P_M(s_0)\prod_{t=0}^{T-1} P_M(a_t\mid s_t)\,P_M(s_{t+1}\mid s_t,a_t),
\end{equation}
at every trajectory $\mathbf v$. When $M\in\cM(P,G,\mathbf S)$, $P_M=P$, so
\eqref{eq:obs_factor} applies with $P_M$ replaced by $P$.
\end{proposition}

\begin{proof}
At time $t$, the fresh exogenous variables $(U_t,\varepsilon_t^A,\varepsilon_{t+1}^S)$
are independent of the observed past through time $t$, by
Assumption~\ref{ass:npc}. The structural equations for $A_t$ and $S_{t+1}$
in~\eqref{eq:scm} depend on that past only through $S_t$. Consequently,
conditional on $S_t$, the law of $(A_t,S_{t+1})$ under $P_M$ does not
depend on the earlier observed history $(S_0,A_0,\ldots,S_{t-1},A_{t-1})$.
If every prefix of $\mathbf v$ has positive $P_M$-probability, applying
the ordinary chain rule to $P_M(\mathbf v)=P_M(s_0)\prod_{t=0}^{T-1}
P_M(a_t,s_{t+1}\mid s_0,a_0,\ldots,s_t)$ and using this conditional
independence at each factor gives $P_M(a_t,s_{t+1}\mid
s_0,a_0,\ldots,s_t)=P_M(a_t,s_{t+1}\mid s_t)$ for every $t$, which is
\eqref{eq:obs_factor}. If instead some prefix of $\mathbf v$ has
probability zero, then $P_M(\mathbf v)=0$, and the null-event convention
above makes the right-hand side of \eqref{eq:obs_factor} $0$ as well.
\end{proof}

\noindent The true structural kernel $\kappa_t(\cdot\mid s,a)$ is
\emph{identified} at a cell with $P(S_t=s,A_t=a)>0$ when it is
recoverable from the observational law there. When $A_t\leftrightarrow S_{t+1}$ is absent from $G$, compatibility
with $G$, together with Assumption~\ref{ass:npc}, implies
\[
\kappa_t(\cdot\mid s,a)=P(S_{t+1}\mid s,a)
\]
at every cell with $P(S_t=s,A_t=a)>0$. Under confounding, by contrast,
$\kappa_t(\cdot\mid s,a)$ need not coincide with $P(S_{t+1}\mid s,a)$,
or be recoverable from the observational law at all.
Section~\ref{sec:feasible} turns this identified/unidentified
distinction, together with the null-cell case, into an explicit
construction of feasible sets.

\section{Local Feasible Sets}\label{sec:feasible}

Section~\ref{sec:model} defined the local structural transition kernel
$\kappa_t(\cdot\mid s,a)$ at every state--action cell
(Definition~\ref{def:kernel}) and showed that it need not be
identified from the observational law under contemporaneous confounding.
This object is defined by the SCM and determines the post-intervention
transition whenever action $a$ is assigned at state $s$, without itself
depending on a particular regime $d$.
We represent what is known about this kernel at each cell (whether that knowledge fully identifies it or only constrains it) by a compact convex
feasible set $\cQ_t(s,a)$ of candidate next-state distributions, built
as a subset of $\Delta(\mathcal S_{t+1}(s,a))$. Throughout this section,
the local structural successor family $\mathbf S$ is fixed as part of
the maintained causal model class. Section~\ref{subsec:default} gives a
default construction from $P$, $G$, and $\mathbf S$;
Section~\ref{subsec:ambiguity} combines the cellwise sets into a single
\emph{interventional ambiguity set} over entire trajectories and shows
that the resulting set is \emph{rectangular}. For the default
construction itself, this ambiguity set is also \emph{sharp}:
Theorem~\ref{thm:sharp} shows that it coincides with the causal
identified set.

\subsection{Local Feasible Sets and the Default Construction}\label{subsec:default}

To place the baseline distribution $P(s_0)$ and the subsequent transition
kernels in a single common framework, it is convenient to adjoin a
deterministic \emph{root state} before $S_0$ is known. We set a dummy state at time $-1$: $\mathcal S_{-1}:=\{*\}$,
$S_{-1}\equiv *$. For $t=0,\ldots,T-1$, we set
$\mathcal A_t:=\cA$, and introduce the dummy action space
$\mathcal A_{-1}:=\{\star\}$. We also use the convention $d_{-1}(*)=\star$. 
For the root transition, we set
$\mathcal S_0(*,\star):=\mathcal S_0$. There is no treatment at time $-1$, and the ``transition''
from the dummy root state-action pair $(*,\star)$ to the initial state $S_0$ is simply the known baseline law
$P(S_0\in\cdot)$. Accordingly, we define the root transition kernel by
\[
\kappa_{-1}(\cdot\mid *,\star):=P(S_0\in\cdot).
\]
This adds no structural or causal assumption; the root state is
purely a bookkeeping device that represents the already-existing factor
$P(s_0)$ in Theorem~\ref{thm:factorization} as a one-step transition, and makes notation more uniform.


\begin{definition}\label{def:lfs}
For $t=-1,0,\ldots,T-1$, $s\in\mathcal S_t$, and
$a\in\mathcal A_t$, a \emph{feasible transition set} is a nonempty
compact convex set
$\cQ_t(s,a)\subseteq\Delta\bigl(\mathcal S_{t+1}(s,a)\bigr)$
satisfying the \emph{consistency condition}
\begin{equation}\label{eq:consistency}
\kappa_t(\cdot\mid s,a)\in\cQ_t(s,a).
\end{equation}
For $t\geq0$, $\kappa_t=\kappa_{M^\star,t}$ is the true SCM's
structural transition kernel
(Definition~\ref{def:kernel} and Remark~\ref{rem:true_scm}); at
$t=-1$, the convention above gives $\kappa_{-1}(\cdot\mid *,\star)=P(S_0\in\cdot)$.
We write
$ \cQ=
\bigl\{\cQ_t(s,a): t=-1,\ldots,T-1,\ s\in\mathcal S_t,\ a\in\mathcal A_t
\bigr\}$
for the full family of feasible transition sets.
\end{definition}

\noindent Given the observational law $P$, the causal graph $G$, and the fixed local structural successor family $\mathbf S$, a feasible set
satisfying Definition~\ref{def:lfs} can be produced by the following
default construction. For the root cell $(-1,*,\star)$ and the
treatment-dependent cells $(t,s,a)$ for $t\geq0$, the construction is as
follows:

\paragraph{Step 1: the root.}
At the root, $S_0$ is exogenous (equation~\eqref{eq:scm}) and hence
unaffected by any treatment or by any $U_t$, so $P^d(S_0)=P(S_0)$ is
always identified; the default construction accordingly sets
\[
\cQ_{-1}^{\mathrm{default}}(*,\star)
:=
\bigl\{P(S_0\in\cdot)\bigr\}.
\]


\paragraph{Step 2: check positivity.}
If $P(S_t=s,A_t=a)=0$, there is no data at this cell to rule out any
next-state distribution, and we set
$\cQ_t^{\mathrm{default}}(s,a)=\Delta(\mathcal{S}_{t+1}(s,a))$, the entire simplex.

\paragraph{Step 3: if positivity holds, check for confounding.}
If $P(s,a)>0$ and $A_t\leftrightarrow S_{t+1}$ is absent from the fixed
graph $G$ (Section~\ref{sec:model}) --- the same fact at every
$s\in\mathcal S_t$ --- then $P(S_{t+1}\mid s,a)$ is a well-defined factor
of the observational law (Proposition~\ref{prop:obs_factorization}) and
$\kappa_t(\cdot\mid s,a)=P(S_{t+1}\mid s,a)$
(Section~\ref{sec:factorization}), and we set
\[
\cQ_t^{\mathrm{default}}(s,a)=\bigl\{P(S_{t+1}\mid S_t=s,A_t=a)\bigr\}.
\]

\paragraph{Step 4: otherwise, construct the Manski polytope.}
If $P(s,a)>0$ and $A_t\leftrightarrow S_{t+1}$ is present, the
observational data identify the mass contributed by units with $A_t=a$,
while the counterfactual next-state distribution of units with $A_t\neq
a$ is,  in principle,  unrestricted over $\mathcal{S}_{t+1}(s,a)$ \citep{Manski1990}. We
therefore set $\cQ_t(s,a)$ to the Manski feasible set
\begin{equation}\label{eq:manski}
\cQ_t^{\mathrm{default}}(s,a)
=\bigl\{q\in\Delta(\mathcal{S}_{t+1}(s,a)):\underline\pi_t(s'\mid s,a)\leq q(s')\leq
\bar\pi_t(s'\mid s,a),\;\forall\,s'\in\mathcal{S}_{t+1}(s,a)\bigr\},
\end{equation}
where, for each $s'\in\mathcal{S}_{t+1}(s,a)$ 
\begin{align}
\underline\pi_t(s'\mid s,a)
  &= P(S_{t+1}=s',A_t=a\mid S_t=s),
  \label{eq:manski_lower}\\
\bar\pi_t(s'\mid s,a)
  &= P(S_{t+1}=s',A_t=a\mid S_t=s)+P(A_t\neq a\mid S_t=s).
  \label{eq:manski_upper}
\end{align}
The lower bound fixes the directly observed mass, while the upper bound
allows all remaining mass to be assigned to a given next state.


\noindent At every cell, $\cQ_t^{\mathrm{default}}(s,a)$ is the sharp identified
set for the single local kernel relative to the local observational
information and the maintained local structural successor set
$\mathcal S_{t+1}(s,a)$: trivially so in Steps~1--3, and by the
classical sharpness of the Manski bounds
\citep{Manski1990,Manski1995} in Step~4. Consequently, any proper
refinement
$\cQ_t(s,a)\subsetneq\cQ_t^{\mathrm{default}}(s,a)$ must encode
additional information or identifying assumptions.
The
full default family is
\[
\cQ^{\mathrm{default}}
=
\bigl\{
\cQ_t^{\mathrm{default}}(s,a):
t=-1,\ldots,T-1,\ 
s\in\mathcal S_t,\ 
a\in\mathcal A_t
\bigr\}.
\]

\begin{lemma}\label{lem:properties}
Under Assumption~\ref{ass:finite}, every default set produced by
Steps~1--4 is a nonempty compact convex polytope. For any closed convex
restrictions $C_1,\ldots,C_J$, intersecting any of these default sets
with $\bigcap_j C_j$ is compact and convex, nonempty if and only if
jointly feasible with the default set, and valid in the sense of
Definition~\ref{def:lfs} only if the true kernel also lies in every
$C_j$.
\end{lemma}

\begin{proof}
Let $X$ denote the relevant finite successor space and $\kappa$ the true
kernel at that cell. By Assumption~\ref{ass:finite}, $X$ is finite, so
$\Delta(X)$ is compact and convex. At positive-probability cells,
Remark~\ref{rem:observed_support_compatibility} ensures that the observed
conditional distribution is supported on the corresponding local
structural successor set. Moreover, since $M^\star\in\cM(P,G,\mathbf S)$, the true structural kernel is supported on that set at every cell.
Each default set is therefore a nonempty intersection of $\Delta(X)$
with finitely many closed half-spaces: Steps~1 and~3 give the identified
singleton, Step~2 gives the whole simplex, and Step~4 imposes the Manski
bounds~\eqref{eq:manski_lower}--\eqref{eq:manski_upper}, which the true
kernel satisfies. Hence every default set is a nonempty compact convex
polytope. For the second: intersecting with closed convex$C_j$'s preserves closedness/convexity, compactness follows since the result is a closed subset of a compact set, nonemptiness is jointfeasibility, and validity needs $\kappa\in\bigcap_j C_j$ since $\kappa$ is already in the default set.
\end{proof}

\begin{remark}\label{rem:flexibility}
The default construction of Section~\ref{subsec:default} provides a
canonical family of feasible sets given the observational law $P$, the
causal graph $G$, and the maintained local structural successor family
$\mathbf S$. At every cell, it returns the sharp identified set for the
local structural kernel relative to the local observational information
and the maintained local structural successor set. This is not the only
way to specify feasible sets, however: the analyst may choose
$\cQ_t(s,a)$ differently depending on their objectives, subject only to
Definition~\ref{def:lfs}'s requirements.
In particular, the default set at a cell
may be narrowed by intersecting it with additional closed convex
restrictions (results from a randomised sub-study,  or a monotonicity or sign restriction on
$\kappa_t(\cdot\mid s,a)$); Lemma~\ref{lem:properties} guarantees the
intersected set remains compact and convex if the restrictions are  
compatible with the true kernel.
Alternatively, the feasible sets may be enlarged deliberately for
sensitivity analysis, and the right device depends on which kind of cell
is being widened.  
At a cell nominally treated as identified (Step 3), a Rosenbaum-style
sensitivity parameter $\Gamma\geq 1$ \citep{Rosenbaum1987} produces a
parametrized family of widening feasible sets as $\Gamma$ increases. 
At a cell already treated as confounded (Step~4), a neighborhood of the Manski
set itself, in the spirit of robust MDPs
\citep{Iyengar2005,NilimElGhaoui2005}, can serve the same purpose. Both types of enlargements provide sensitivity analyses by allowing a
wider range of candidate kernels while retaining the same local
structural successor sets. The root set $\cQ_{-1}(*,\star)$ also admits the same widening operation.
\end{remark}


\subsection{The Interventional Ambiguity Set and Rectangularity}\label{subsec:ambiguity}

\begin{definition}\label{def:ambiguity}
Fix $d\in\cD$ and a family
\[
\cQ=
\bigl\{
\cQ_t(s,a):
t=-1,\ldots,T-1,\ 
s\in\mathcal S_t,\ 
a\in\mathcal A_t
\bigr\}
\]
of feasible transition sets as in Definition~\ref{def:lfs}. The
\emph{interventional ambiguity set} $\cP^d(\cQ)$ is the collection of
all distributions $Q$ on $\mathbf V$ of the form
\begin{equation}\label{eq:ambiguity}
Q(\mathbf v)
=
q_{-1}(s_0\mid *,\star)
\prod_{t=0}^{T-1}
\ind\{a_t=d_t(s_t)\}\,
q_t(s_{t+1}\mid s_t,a_t),
\end{equation}
where, for every $t=-1,\ldots,T-1$ and $s\in\mathcal S_t$,
\[
q_t(\cdot\mid s,d_t(s))
\in
\cQ_t(s,d_t(s))
\subseteq
\Delta\bigl(\mathcal S_{t+1}(s,d_t(s))\bigr).
\]
\end{definition}

\noindent We suppress the argument $\cQ$ ($\cP^d(\cQ) \equiv \cP^d$) when the family of feasible sets is fixed,
and restore it when comparing different families. By the consistency
condition~\eqref{eq:consistency} and Theorem~\ref{thm:factorization}, the
true interventional law $P^d$ belongs to $\cP^d(\cQ)$. The independent
choice of one feasible kernel at each cell, including the root, gives the
rectangular structure that we formalize next.

\begin{proposition}[Rectangularity]\label{prop:rect}
$\cP^d(\cQ)$ is \emph{rectangular}: replacing the
kernel at any cell $(t,s,d_t(s))$, $t=-1,\ldots,T-1$, by any other
kernel from the same feasible set, while keeping all other kernels
fixed, yields another element of $\cP^d(\cQ)$.
\end{proposition}

\begin{proof}
Immediate from Definition \ref{def:ambiguity}.
\end{proof}

\noindent This is the $(s,a)$-rectangularity condition of \citet{Iyengar2005}; it
permits the trajectory-level optimisation to decompose cell by cell in
Section~\ref{sec:bounds}.

\begin{definition}[Causal identified set]\label{def:identified_set}
Fix $d\in\cD$. The \emph{causal identified set} for $P^d$ is
\[
\cI^d(P,G,\mathbf S)
:=
\{P_M^d:M\in\cM(P,G,\mathbf S)\},
\]
where $\cM(P,G,\mathbf S)$ is the admissible SCM class of
Definition~\ref{def:admissible_class}.
\end{definition}

\begin{remark}[Structural successor sets and feasible-set modifications] \label{rem:rect_justification}
The local structural successor family $\mathbf S$ is part of the
maintained SCM class. Consequently, enlarging $\mathbf S$ weakens the
structural model and may enlarge both $\cM(P,G,\mathbf S)$ and the
causal identified set $\cI^d(P,G,\mathbf S)$. This is distinct from modifying the feasible sets $\cQ_t(s,a)$ while
holding $P$, $G$, and $\mathbf S$ fixed. Theorem~\ref{thm:sharp} below
shows that the unmodified default family yields exactly
$\cI^d(P,G,\mathbf S)$. Narrowed, enlarged, or statistical feasible
sets remain legitimate rectangular ambiguity classes for the dynamic
program, but need not coincide with the causal identified set.
\end{remark}

\begin{theorem}[Exact Sharpness under Contemporaneous Confounding]%
\label{thm:sharp}
Fix $d\in\cD$ and a local structural successor family $\mathbf S$
satisfying Assumption~\ref{ass:finite}, and let
$\cQ^{\mathrm{default}}$ be the default family constructed from
$P$, $G$, and $\mathbf S$. Under the maintained
model class ~\eqref{eq:scm} and Assumptions~\ref{ass:npc}
and~\ref{ass:modularity},
\[
\cP^d(\cQ^{\mathrm{default}})
=
\cI^d(P,G,\mathbf S).
\]
\end{theorem}

\begin{proof}
See Appendix~\ref{app:thm_sharp}.
\end{proof}

\noindent The equality is specific to the unmodified default family. For narrowed or
enlarged feasible sets (Remark ~\ref{rem:flexibility}), the dynamic program remains exact for the
specified rectangular class $\cP^d(\cQ)$ (Remark~\ref{cor:sharp_class}),
but that class need not coincide with the causal identified set.

\section{Causal Bounds and Optimal Regime Selection via Backward Induction}\label{sec:bounds}

Definition~\ref{def:lfs} in Section~\ref{sec:feasible} attached a feasible set $\cQ_t(s,a)$ to every
cell $(t,s,a)$. For a given regime $d$, different feasible choices of
the transition kernels may produce different expected outcomes. In
Subsection~\ref{subsec:evaluation} we compute the smallest and largest of
these outcomes by backward induction. To choose among regimes, we must then compare the resulting ranges. In
Subsection~\ref{subsec:selection} we use the maximin criterion, choosing
the regime with the largest worst-case expected outcome.

\subsection{Regime Evaluation: Causal Bounds and Backward Induction}\label{subsec:evaluation}

For a given regime $d$, the interventional ambiguity set
$\cP^d(\cQ)$, defined in Definition~\ref{def:ambiguity}, consists of the
trajectory laws obtained by choosing the relevant transition kernels from
their feasible sets. We now define the smallest and largest expected
outcomes over this set:

\begin{definition}[Causal bounds]\label{def:bounds}
For a family $\cQ$ of local feasible sets,
\[
\bar\theta^d(\cQ)=\sup_{Q\in\cP^d(\cQ)}\E^Q[Y],\qquad
\underline\theta^d(\cQ)=\inf_{Q\in\cP^d(\cQ)}\E^Q[Y].
\]
As with $\cP^d$, we write $\bar\theta^d,\underline\theta^d$ for
$\bar\theta^d(\cQ),\underline\theta^d(\cQ)$ when $\cQ$ is fixed.
\end{definition}
\noindent Below we rely on the notation $V^d = \E[Y^d]$.
\begin{proposition}\label{prop:basic}
Under Assumptions~\ref{ass:modularity}--\ref{ass:finite}:
$(a)$ $\bar\theta^d$ and $\underline\theta^d$ are finite and attained;
$(b)$ $\underline\theta^d\leq V^d\leq\bar\theta^d$;
$(c)$ when all feasible sets are singletons,
$\bar\theta^d=\underline\theta^d=V^d$.
\end{proposition}
\begin{proof}
For (a), Assumption~\ref{ass:finite} makes the state and action spaces
finite, and hence there are only finitely many relevant cells. Moreover,
$Y$ is bounded since $\mathcal Y$ is finite. By
Definition~\ref{def:lfs}, each $\cQ_t(s,a)$ is compact. Therefore
\[
\prod_{t=-1}^{T-1}\prod_{s\in\mathcal S_t}
\cQ_t(s,d_t(s))
\]
is compact. By Definition~\ref{def:ambiguity}, each element of this
product determines a trajectory law $Q\in\cP^d$, and $\E^Q[Y]$ is a
finite sum of products of the selected kernel probabilities. It is
therefore a continuous function of those kernel probabilities. Hence its
maximum and minimum are attained, proving (a).
For (b), the consistency condition~\eqref{eq:consistency}, together with
Theorem~\ref{thm:factorization}, gives $P^d\in\cP^d$. Thus
\[
\underline\theta^d\leq \E^{P^d}[Y]=V^d\leq\bar\theta^d.
\]
For (c), if every feasible set is a singleton, consistency forces its
unique element to be the corresponding true kernel, 
including the kernel $\kappa_{-1}(\cdot\mid *,\star)=P(S_0\in\cdot)$
Hence $\cP^d=\{P^d\}$, so
$\bar\theta^d=\underline\theta^d=V^d$.
\end{proof}

\noindent Rather than optimise over all trajectory laws in $\cP^d(\cQ)$ at once,
we can use rectangularity (Proposition~\ref{prop:rect}) to work
backwards, one time step at a time. For this purpose we define the
possible future laws and expected outcomes starting from any state
$s\in\mathcal S_t$, including states that have zero probability under
the baseline law.

\begin{definition}[Continuation ambiguity sets and value functions]
\label{def:vf}
Fix $d\in\cD$ and a family $\cQ$ of feasible transition sets. For
$t\in\{-1,0,\ldots,T-1\}$ and $s\in\mathcal S_t$, let
$\cP_t^d(s;\cQ)$ be the collection of continuation laws of
$(S_{t+1},\ldots,S_T)$, initiated from $S_t=s$, of the form
\begin{equation}\label{eq:continuation}
Q(s_{t+1},\ldots,s_T)
=
\prod_{k=t}^{T-1}
q_k\bigl(s_{k+1}\mid s_k,d_k(s_k)\bigr),
\end{equation}
where $s_t:=s$ and, for every $k=t,\ldots,T-1$ and
$x\in\mathcal S_k$,
\[
q_k(\cdot\mid x,d_k(x))
\in
\cQ_k(x,d_k(x)).
\]
Thus the continuation law is defined directly from the
selected kernels starting at the fixed state $s$, rather than by
conditioning a pre-existing joint law on $\{S_t=s\}$. It is therefore
well defined even when $s$ has zero probability.
At the terminal time, set
\[
\cP_T^d(s;\cQ):=\{\delta_s\}.
\]
For $t=-1,0,\ldots,T$ and $s\in\mathcal S_t$, define
\[
\overline V_t^d(s;\cQ)
:=
\sup_{Q\in\cP_t^d(s;\cQ)}\E^Q[Y],
\qquad
\underline V_t^d(s;\cQ)
:=
\inf_{Q\in\cP_t^d(s;\cQ)}\E^Q[Y].
\]
When $\cQ$ is fixed, we write
$\overline V_t^d(s)$ and $\underline V_t^d(s)$.
\end{definition}
\noindent At the root, the continuation law has the form
\[
Q(s_0,\ldots,s_T)
=
q_{-1}(s_0\mid *,\star)
\prod_{k=0}^{T-1}
q_k\bigl(s_{k+1}\mid s_k,d_k(s_k)\bigr).
\]

\noindent For the same choice of kernels, Definition~\ref{def:ambiguity} gives a
unique law in $\cP^d(\cQ)$ whose state marginal is precisely this
continuation law: the actions satisfy $A_t=d_t(S_t)$ deterministically.
Conversely, any law in $\cP_{-1}^d(*;\cQ)$ determines a unique law in
$\cP^d(\cQ)$ by adjoining these deterministic actions and
$Y=y(S_T)$.
Thus $\cP_{-1}^d(*;\cQ)$ is in one-to-one correspondence with
$\cP^d(\cQ)$ through the state marginal. Since $Y=y(S_T)$,
corresponding laws have the same distribution of $Y$ and therefore the
same expected outcome. Hence, directly from Definition~\ref{def:bounds},
\[
\overline V_{-1}^d(*;\cQ)=\bar\theta^d(\cQ),\qquad
\underline V_{-1}^d(*;\cQ)=\underline\theta^d(\cQ).
\]

\begin{theorem}[Backward induction]\label{thm:dp}
Fix a family $\cQ$ of feasible transition sets. The value functions of
Definition~\ref{def:vf} are computed exactly by the following backward
recursion.
At $t=T$,
\[
\overline V_T^d(s)
=
\underline V_T^d(s)
=
y(s),
\qquad s\in\mathcal S_T.
\]
For $t=T-1,\ldots,-1$ and every $s\in\mathcal S_t$,
\begin{align}
\overline V_t^d(s)
&=
\max_{q\in\cQ_t(s,d_t(s))}
\sum_{s'\in\mathcal S_{t+1}(s,d_t(s))}
q(s')\,\overline V_{t+1}^d(s'),
\label{eq:upper_rec}\\[3pt]
\underline V_t^d(s)
&=
\min_{q\in\cQ_t(s,d_t(s))}
\sum_{s'\in\mathcal S_{t+1}(s,d_t(s))}
q(s')\,\underline V_{t+1}^d(s').
\label{eq:lower_rec}
\end{align}
Consequently,
\[
\overline V_{-1}^d(*)=\bar\theta^d,
\qquad
\underline V_{-1}^d(*)=\underline\theta^d.
\]
Under the default construction,
$\cQ_{-1}(*,\star)=\{P(S_0\in\cdot)\}$, so the $t=-1$ instance of the
recursion reduces to
\[
\overline V_{-1}^d(*)
=
\sum_{s_0\in\mathcal S_0}P(s_0)\overline V_0^d(s_0),
\qquad
\underline V_{-1}^d(*)
=
\sum_{s_0\in\mathcal S_0}P(s_0)\underline V_0^d(s_0).
\]
\end{theorem}

\begin{proof}
We prove the upper recursion; the lower recursion is identical with
maxima replaced by minima throughout. 
For $t=-1,\ldots,T-1$ write $\cQ_t^d(s):=\cQ_t(s,d_t(s))$, 
and $q_t(\cdot\mid s):=q_t(\cdot\mid s,d_t(s))$. When convenient,
extend this distribution by zero from
$\mathcal S_{t+1}(s,d_t(s))$ to the whole state space
$\mathcal S_{t+1}$.
Every $Q\in\cP_t^d(s;\cQ)$
is built from a \emph{continuation family}
$q=\{q_k(\cdot\mid x,d_k(x))\in\cQ_k(x,d_k(x))\}_{k=t,\ldots,T-1;\,x\in\mathcal
S_k}$; write $\E_q^{t,s}[\,\cdot\,]$ for the expectation under the resulting
law, so that $\overline V_t^d(s;\cQ)=\sup_q\E_q^{t,s}[y(S_T)]$. By the
product form~\eqref{eq:continuation}, having stepped to $S_{t+1}=s'$ the
remaining evolution under $q$ is governed by exactly the same tail of
$q$, so for $t=-1,\ldots,T-1$
\begin{equation}\label{eq:tower_q}
\E_q^{t,s}[y(S_T)]=\sum_{s'\in\mathcal S_{t+1}}q_t(s'\mid s)\,\E_q^{t+1,s'}[y(S_T)].
\end{equation}
We argue by downward induction on $t=T,T-1,\ldots,-1$, carrying the
strengthened claim: \emph{there is a continuation family $q^*$ from $t$
such that $\E_{q^*}^{\tau,x}[y(S_T)]=\overline V_\tau^d(x)$ for every
$\tau\geq t$ and every $x\in\mathcal S_\tau$.} At $t=T$ this is immediate,
since $\cP_T^d(x;\cQ)=\{\delta_x\}$ gives $\overline V_T^d(x)=y(x)$ for
every $x\in\mathcal S_T$.

\noindent Suppose the claim holds at $t+1$, with attaining family $q^*$, and fix
$s\in\mathcal S_t$. For \emph{any} continuation family $q$ from $(t,s)$,
\eqref{eq:tower_q} together with $\E_q^{t+1,s'}[y(S_T)]\leq\overline
V_{t+1}^d(s')$ --- valid by definition of $\overline V_{t+1}^d$ as a
supremum, regardless of whether $q$'s tail is $q^*$ --- gives
\begin{equation}\label{eq:one_bound}
\E_q^{t,s}[y(S_T)]=\sum_{s'}q_t(s'\mid s)\,\E_q^{t+1,s'}[y(S_T)]
\;\leq\;\sum_{s'}q_t(s'\mid s)\,\overline V_{t+1}^d(s')
\;\leq\;\max_{r\in\cQ_t^d(s)}\sum_{s'}r(s')\,\overline V_{t+1}^d(s'),
\end{equation}
the last step because $q_t(\cdot\mid s)\in\cQ_t^d(s)$. Taking the
supremum over $q$ on the left shows $\overline V_t^d(s)$ is bounded above
by the maximum on the right, which is attained --- $\cQ_t^d(s)$ is
compact and $r\mapsto\sum_{s'}r(s')\overline V_{t+1}^d(s')$ is linear,
hence continuous. Fix a maximizer $r_t^*(\cdot\mid s)$.

\noindent It remains to show this upper bound is attained, which will give equality
directly, with no separate lower bound to prove. Let $\tilde q$ be
$r_t^*(\cdot\mid s)$ prepended to $q^*$: a continuation family from
$(t,s)$ whose tail from $t+1$ is exactly $q^*$. By~\eqref{eq:tower_q},
and because the induction hypothesis gives
$\E_{q^*}^{t+1,s'}[y(S_T)]=\overline V_{t+1}^d(s')$ at every $s'$,
\[
\E_{\tilde q}^{t,s}[y(S_T)]=\sum_{s'}r_t^*(s'\mid s)\,\E_{q^*}^{t+1,s'}[y(S_T)]
=\sum_{s'}r_t^*(s'\mid s)\,\overline V_{t+1}^d(s')
=\max_{r\in\cQ_t^d(s)}\sum_{s'}r(s')\,\overline V_{t+1}^d(s'),
\]
so $\tilde q$ attains the bound~\eqref{eq:one_bound}. Since
$\overline V_t^d(s)=\sup_q\E_q^{t,s}[y(S_T)]\geq\E_{\tilde
q}^{t,s}[y(S_T)]$, combining with the upper bound already established
gives
\[
\overline V_t^d(s)=\max_{r\in\cQ_t^d(s)}\sum_{s'}r(s')\,\overline V_{t+1}^d(s'),
\]
which is~\eqref{eq:upper_rec} for $t=-1,\ldots,T-1$

\noindent Rectangularity (Proposition~\ref{prop:rect}) permits choosing
$r_t^*(\cdot\mid x)$ this way independently at every $x\in\mathcal S_t$,
root included: this extends $q^*$ to a family valid from $t$, attaining
$\overline V_\tau^d(x)$ for every $\tau\geq t$ and every $x$, since
$\E_{q^*}^{\tau,x}$ depends only on $q^*$'s tail from $\tau$, unaffected
by prepending the new kernels at $t$. This completes the induction down
to $t=-1$. 
\end{proof}


\begin{remark}[Properties of the backward induction]
Comments on  Theorem~\ref{thm:dp}.

\emph{(a)}\label{cor:sharp_class}\label{rem:two_claims}
Theorem~\ref{thm:dp} holds for \emph{any} family $\cQ$ of local feasible
sets satisfying Definition~\ref{def:lfs} (default, narrowed, or enlarged):
it returns the exact bounds
$[\underline\theta^d(\cQ),\bar\theta^d(\cQ)]$ for
$\cP^d(\cQ)$ regardless of how $\cQ$ was chosen. This is a statement
about the dynamic program and the specified rectangular ambiguity class.
When $\cP^d(\cQ)$ coincides with the causal identified set
$\cI^d(P,G,\mathbf S)$, the resulting bounds are also sharp for the
causal model. By Theorem~\ref{thm:sharp}, this holds for the
unmodified default family constructed from $P$, $G$, and $\mathbf S$
(Remark~\ref{rem:rect_justification}).

\emph{(b)}\label{rem:cell_certificates} At no additional cost,  this recursion 
produces the conditional intervals
$[\underline V_t^d(s),\overline V_t^d(s)]$ at every cell.
A degenerate interval, $\overline V_t^d(s)=\underline
V_t^d(s)$, means every continuation law in $\cP_t^d(s;\cQ)$ gives the
same remaining expected outcome from $(t,s)$. Similar conclusions hold
when the width is small. These widths are therefore informative about
where uncertainty in the expected outcome under the given regime remains:
a cell with no width has an already determined continuation value, even
if its underlying transition kernel is not itself identified.
To our knowledge, existing partial-identification and DTR methods do not
offer a general cell-level certificate of this kind. The granularity is
close in spirit to \citet{ChenDarwiche2025}, who show a state-specific
causal effect can be identified even when the corresponding variable-level
effect is not; here the same phenomenon occurs across cells of a single
regime's value tree,  with full identification of the continuation value at one $(t,s)$ and
not at another.

\emph{(c)}\label{rem:local_compress} At each cell
$(t,s)$ the algorithm solves a single linear programme over
$\cQ_t(s,d_t(s))$; no global optimisation over transition sequences is
required, and singleton feasible sets introduce no additional local kernel
ambiguity, although the continuation value may still have width because
of uncertainty present at later steps.
When confounded or unsupported cells are numerous, the method is
naturally interpreted as a sensitivity analysis rather than expected to
produce narrow bounds.
\end{remark}

\noindent A useful check on Theorem~\ref{thm:dp} is that it must reduce to something
familiar once ambiguity is switched off. Suppose $S_t=\bar H_t=(L_0,A_0,
\ldots,L_t)$ is the full observed history in the classical framework of \citet{HernanRobins2020}, and suppose sequential exchangeability \emph{and} positivity hold at every history-action cell
relevant under the regime $d$: for every $t$ and every history $\bar h_t$
that can be reached under $d$,
\[
\cQ_t(\bar h_t,d_t(\bar h_t))=\bigl\{P(L_{t+1}\mid \bar H_t=\bar h_t,
A_t=d_t(\bar h_t))\bigr\}.
\]
For a fixed regime $d$, positivity is needed only for the action selected
by $d$ at each relevant history. In the Q-learning reduction below
(Corollary~\ref{cor:qlearning}), positivity is instead needed for every
action being compared.

Write $\bar\ell_t=(l_0,\ldots,l_t)$ and
$\bar a_t=(a_0,\ldots,a_t)$, where
$a_k=d_k(\bar h_k)$ for $k=0,\ldots,t$.
\begin{corollary}[The g-formula as a special case]\label{cor:gformula}
Suppose sequential exchangeability \emph{and} positivity hold at every history-action cell
relevant under the regime $d$. Under the singleton condition above, together with the default root set
$\cQ_{-1}(*,\star)=\{P(L_0\in\cdot)\}$,
\begin{equation}\label{eq:gformula_val}
\bar\theta^d=\underline\theta^d=V^d
=\sum_{\bar\ell_{T}}y(\bar l_T)\,P(l_0)
\prod_{t=0}^{T-1}P(l_{t+1}\mid\bar\ell_t,\bar a_t).
\end{equation}
\end{corollary}
\noindent Equation~\eqref{eq:gformula_val} is the nonparametric g-formula
of \citet{Robins1986} in the history-based formulation of
\citet{HernanRobins2020}.
\begin{proof}
Singleton sets make the $\max$ and $\min$ in Theorem~\ref{thm:dp} trivial,
collapsing each recursion step to
$\overline V_t^d(\bar h_t)=\underline V_t^d(\bar h_t)
=\sum_{l'}P(l'\mid\bar h_t,d_t(\bar h_t))\,\overline V_{t+1}^d(\bar h_t,d_t(\bar h_t),l')$.
Under sequential exchangeability and positivity, for $t\geq0$, the structural kernel
 $\kappa_t(\cdot\mid\bar h_t,a_t)
=P(L_{t+1}\in\cdot\mid\bar H_t=\bar h_t,A_t=a_t)$ 
while the convention at $t=-1$ gives $\kappa_{-1}(l_0\mid *,\star)=P(l_0)$.
Consequently, iterating the common recursion of Theorem~\ref{thm:dp}
from $t=T-1$ through $t=-1$ gives
\[
\overline V_{-1}^d(*)
=
\underline V_{-1}^d(*)
=
\sum_{\bar\ell_T}
y(\bar\ell_T)\,P(l_0)
\prod_{t=0}^{T-1}
P(l_{t+1}\mid\bar\ell_t,\bar a_t),
\qquad
\bar a_t=d_t(\bar h_t).
\]
The left-hand side equals
$\bar\theta^d=\underline\theta^d$. Proposition~\ref{prop:basic}(c)
then gives the common value $V^d$, proving
equation~\eqref{eq:gformula_val}.
\end{proof}

\begin{remark}[Financial interpretation]\label{rem:tree_methods}
Theorem~\ref{thm:dp} has a natural reading in the language of incomplete
financial markets. For a fixed regime $d$, the terminal outcome
$Y=y(S_T)$ plays the role of a contingent claim, and $\cP^d(\cQ)$ plays
the role of the set of pricing measures consistent with the local
feasible sets; the global bounds $\underline\theta^d$ and $\bar\theta^d$
are then analogous to the bid and ask prices of the claim. The same
reading applies at every intermediate node: $\underline V_t^d(s)$ and
$\overline V_t^d(s)$ are the bid and ask prices of the remaining claim
conditional on being at $s$, coinciding with the payoff at maturity,
$\underline V_T^d(s)=\overline V_T^d(s)=y(s)$, and propagating backward
through the feasible sets, as in Theorem~\ref{thm:dp}, down to
$\underline V_{-1}^d(*)=\underline\theta^d$ and
$\overline V_{-1}^d(*)=\bar\theta^d$ at the root. Backward propagation of
bid and ask prices through an outcome tree is a standard pricing technique in mathematical finance.
\end{remark}

\subsection{Optimal Regime Selection under Ambiguity}\label{subsec:selection}\label{sec:optimal}

Regime evaluation bounds $V^d$ for a fixed $d$; regime \emph{selection}
requires an additional criterion, since partially identified value
intervals may overlap across regimes and ``$\argmax_d V^d$'' is not
well-defined. A plug-in ranking based only on the observational kernel can
ignore the relevant ambiguity and reverse the preferred regime
(Section~\ref{subsec:biom}); a subjective prior reintroduces the
unmodelled assumption the feasible-set construction was built to avoid;
minimax regret couples a regime's evaluation to every other regime's
value, breaking the rectangular decomposition used below.

\noindent We adopt instead the \emph{maximin} criterion,
$d^*_{\mathrm{mm}}\in\argmax_{d}\underline\theta^d$: a regime with the
best worst-case lower bound. This is a conservative, ambiguity-averse
decision rule in the tradition of Wald and
\citet{GilboaSchmeidler1989}, and, under $(s,a)$-rectangularity
(Proposition~\ref{prop:rect}), it is computable by the same robust
Bellman recursion studied in the \emph{robust Markov decision process}
literature \citep{Iyengar2005,NilimElGhaoui2005}, rather than requiring
bespoke machinery.

\begin{definition}[Maximin optimal DTR]\label{def:maximin}
A regime $d^*_{\mathrm{mm}}\in\cD$ is \emph{maximin optimal} if
\[d^*_{\mathrm{mm}}\in\argmax_{d\in\cD}\,\underline\theta^d
=\argmax_{d\in\cD}\,\inf_{Q\in\cP^d}\E^Q[Y].\]
\end{definition}

Rectangularity again converts the trajectory-level maximin problem into a
backward recursion, now with a $\max$ over actions interleaved with the
worst-case $\min$ over kernels at every step.  The root-state and
dummy-action conventions introduced in Subsection~\ref{subsec:default}
allow the same recursion to run uniformly from $t=T-1$ down to $t=-1$.
At $t=-1$, the action set $\mathcal A_{-1}=\{\star\}$ is a singleton, so
the action maximisation introduces no additional choice.


\begin{theorem}[Maximin backward induction]\label{thm:optimal}
Fix a family
\[
\cQ=
\bigl\{
\cQ_t(s,a):
t=-1,\ldots,T-1,\ 
s\in\mathcal S_t,\ 
a\in\mathcal A_t
\bigr\}
\]
of local feasible sets. Define $W_T(s)=y(s)$ and, for
$t=T-1,\ldots,-1$,
\begin{equation}\label{eq:minimax_rec}
W_t(s)
=
\max_{a\in\mathcal A_t}
\min_{q\in\cQ_t(s,a)}
\sum_{s'\in\mathcal S_{t+1}(s,a)}
q(s')\,W_{t+1}(s').
\end{equation}
Then, for $t\geq0$,
\[
W_t(s)
=
\max_{d_{t:T-1}}
\underline V_t^d(s;\cQ),
\qquad
d_{t:T-1}=(d_t,\ldots,d_{T-1}),
\]
and any selectors
\begin{equation}\label{eq:optimal_action}
d_t^*(s)
\in
\argmax_{a\in\cA}
\min_{q\in\cQ_t(s,a)}
\sum_{s'\in\mathcal S_{t+1}(s,a)}
q(s')\,W_{t+1}(s'),
\qquad t=0,\ldots,T-1,
\end{equation}
form a maximin-optimal regime. In particular,
$\displaystyle{
W_{-1}(*)
=
\max_{d\in\cD}\underline\theta^d
=
\underline\theta^{d^*}}.
$\\
Under the default root set
$\cQ_{-1}(*,\star)=\{P(S_0\in\cdot)\}$,
$\displaystyle{
W_{-1}(*)
=
\sum_{s_0\in\mathcal S_0}P(s_0)W_0(s_0)}.
$
\end{theorem}


\begin{proof}
Equation~\eqref{eq:minimax_rec}, restricted to $t\geq0$, is the
finite-horizon robust Bellman recursion of \citet{Iyengar2005}: finite
state and action spaces, compact convex feasible sets $\cQ_t(s,a)$ whose
choices at distinct state-action cells are independent by rectangularity
(Proposition~\ref{prop:rect}), and no running reward beyond the terminal
payoff $y$ --- so both the recursion and the optimality of the greedy
Markov policy $d^*$ follow directly: $W_t(s)=\underline V_t^{d^*}(s)$ for
every $t,s$, where $\underline V_t^d(s;\cQ)$ depends on $d$ only through
$d_{t:T-1}$ (Definition~\ref{def:vf}). The same recursion at $t=-1$ has
only the singleton dummy action $\star$, so
$W_{-1}(*)=\min_{q\in\cQ_{-1}(*,\star)}\sum_{s_0}q(s_0)\,W_0(s_0)$; and
Theorem~\ref{thm:dp} applied to $d^*$ gives
$\underline\theta^{d^*}=\min_{q\in\cQ_{-1}(*,\star)}\sum_{s_0}q(s_0)\,\underline
V_0^{d^*}(s_0)=\min_{q\in\cQ_{-1}(*,\star)}\sum_{s_0}q(s_0)\,W_0(s_0)=W_{-1}(*)$,
with $\underline\theta^{d^*}=\max_{d\in\cD}\underline\theta^d$ by the
maximin optimality of $d^*$ just established. The singleton-root formula
is immediate from $\cQ_{-1}(*,\star)=\{P(S_0\in\cdot)\}$.
\end{proof}

\begin{remark}
As said before,  equation~\eqref{eq:minimax_rec} is exactly the Bellman equation of a
robust MDP \citep{Iyengar2005,NilimElGhaoui2005}. The present framework
contributes a new \emph{source} for its inputs (the uncertainty sets
$\{\cQ_t(s,a)\}$ are derived from the observational law and causal graph
of Sections~\ref{sec:model}--\ref{sec:feasible} rather than posited
exogenously) not a new algorithm: the robust-control machinery is
reused unchanged.
\end{remark}

\noindent The reduction to a familiar classical procedure again serves as a sanity
check. If sequential exchangeability \emph{and positivity} hold everywhere
--- $P(H_t=h_t,A_t=a)>0$ for every history-action pair under
consideration, so that $\cQ_t(s,a)$ is the observational singleton
$\{P(S_{t+1}\mid s,a)\}$ at \emph{every} $(t,s,a)$ the recursion
evaluates, not merely those with observed support --- the worst-case
$\min$ in \eqref{eq:minimax_rec} has nothing left to range over at any
action, and the maximin recursion collapses to the ordinary backward induction structure underlying Q-learning:

\begin{corollary}[Reduction to Q-learning]\label{cor:qlearning}
Suppose sequential exchangeability and positivity hold for every
state--action cell evaluated by the maximisation: for every relevant $t$,
$s$, and every \emph{candidate} action $a\in\cA$,
$P(S_t=s,A_t=a)>0$, so that $\cQ_t(s,a)=\{P(S_{t+1}\mid s,a)\}$ at every
such cell. Then \eqref{eq:minimax_rec} reduces to
\begin{equation}\label{eq:qlearning}
Q_t(s,a)=\sum_{s'\in\mathcal{S}_{t+1}(s,a)}P(S_{t+1}=s'\mid s,a)\,W_{t+1}(s'),
\qquad
W_t(s)=\max_{a\in\cA}Q_t(s,a),
\end{equation}
and the greedy policy $d^*_t(s)\in\argmax_a Q_t(s,a)$
is the standard Q-learning optimal policy. In particular, with
$S_t=\bar H_t=(L_0,A_0,\ldots,L_t)$, this $Q_t$ is exactly the
action-value function of \citet{Murphy2003} in their original
history-based formulation, satisfying the recursion
\[
Q_t(h_t,a_t)=\E\!\left[\max_{a_{t+1}}Q_{t+1}(H_{t+1},a_{t+1})\;\Big|\;
H_t=h_t,\,A_t=a_t\right],
\]
obtained by substituting $W_{t+1}(\cdot)=\max_{a_{t+1}}Q_{t+1}(\cdot,a_{t+1})$
into \eqref{eq:qlearning}. Murphy's Q-function is thus $Q_t$, not $W_t$; the
latter is the corresponding optimal \emph{value} function,
$W_t(h_t)=\max_{a_t}Q_t(h_t,a_t)$.
\end{corollary}

\begin{proof}
With $\cQ_t(s,a)$ a singleton for every $a\in\cA$, not merely the
observed ones, the $\min_q$ in \eqref{eq:minimax_rec} is trivial at every
term of the $\max_a$, giving the first identity in \eqref{eq:qlearning};
the second, $W_t(s)=\max_a Q_t(s,a)$, is then immediate from
\eqref{eq:minimax_rec}. Setting $S_t=\bar H_t$ and substituting
$W_{t+1}=\max_{a_{t+1}}Q_{t+1}(\cdot,a_{t+1})$ into the definition of
$Q_t$ identifies $Q_t$ with Murphy's Q-function by the tower property.
\end{proof}


\section{Numerical Examples}\label{sec:numerical}

\subsection{HIV Antiretroviral Treatment: A Classical Example
            Extended}\label{subsec:hiv}

\citet{Naimi2017} provide a canonical worked illustration of the g-formula
for a two-period HIV treatment problem, assuming sequential exchangeability
at both transitions. We use the same numerical data, but relax the sequential
exchangeability assumption at the first transition, so the first transition kernel is 
not identified; the second transition retains its identified
kernel.  We make no claim that this hybrid assumption is more or less appropriate for
this hypothetical HIV example; the purpose of the example is illustrative.

\noindent A hypothetical cohort of $N=1{,}000{,}000$ HIV-positive patients is followed over two
treatment periods, as discussed below. At baseline all patients
have elevated viral load ($Z_0=1$). The temporal variables are:
$A_0\in\{0,1\}$ (whether treatment was provided at visit~0);
$Z_1\in\{0,1\}$ (viral load at visit~1; $0=$ low $\leq\!200$ copies/ml);
$A_1\in\{0,1\}$ (whether treatment was provided at visit~1);
$C\in\R$ (CD4 lymphocyte count, cells/mm$^3$), the outcome-bearing
variable.

\noindent Set $S_0=Z_0$. The full histories at times~1 and~2 are
$\bar H_1=(S_0,A_0,Z_1)$ and $\bar H_2=(S_0,A_0,Z_1,A_1,C)$.  Since
$S_0\equiv1$ almost surely (every patient begins with elevated viral
load), it carries no information,  so we omit it from both variables,  and consider instead
the trimmed variables  $S_1=(A_0,Z_1)$ and $S_2=(A_0,Z_1,A_1,C)$.  The outcome is
$Y=y(S_2)$, with $y$ the map extracting the last coordinate,
$y(a_0,z_1,a_1,c)=c$, therefore $Y=C$. Table~\ref{tab:hiv_obs} reproduces Table~1 of \citet{Naimi2017}.\footnote{%
\citet{Naimi2017} describe their cohort as $N=1{,}000{,}000$, but the
eight cell counts printed in their Table~1 sum to only $999{,}997$; their
web supplement gives $60{,}657$, not $60{,}654$, for the
$(A_0,Z_1,A_1)=(0,1,0)$ cell, resolving the shortfall. We use this
corrected count, so that our total is $N=1{,}000{,}000$.}

\noindent Let $\mathcal C$ denote the finite
measurement support of the CD4 count $C$. We take the local structural
successor sets to be
\[
\mathcal S_1(1,a_0)
=
\{(a_0,0),(a_0,1)\},
\qquad a_0\in\{0,1\},
\]
and
\[
\mathcal S_2((a_0,z_1),a_1)
=
\{(a_0,z_1,a_1,c):c\in\mathcal C\}.
\]
Thus, at each transition, the local structural successor set consists
precisely of the trimmed histories that validly extend the current
state--action pair.

\noindent The causal graph (Figure~\ref{fig:dag}'s directed skeleton) has directed
edges $S_0\to A_0$, $S_0\to S_1$, $A_0\to S_1$, $S_1\to A_1$, $S_1\to S_2$,
$A_1\to S_2$.  We allow a latent common cause of $A_0$ and the $Z_1$
component of $S_1=(A_0,Z_1)$,  thereby the graph contains the bidirected edge
$A_0\leftrightarrow S_1$ 
representing possible unmeasured confounding of treatment initiation.
No $A_1\leftrightarrow S_2$ edge is assumed, matching the Naimi
sequential-exchangeability assumption at step~1, so all step-1 feasible
sets are singletons.

Although we write $C\in\mathbb R$ for convenience, $C$ is a measured
quantity (a CD4 count from a laboratory test) and therefore has a
natural finite measurement support $\mathcal C$, determined by the
assay's finite resolution and reporting range. Thus the notation
$C\in\mathbb R$ is only an idealization and does not represent a
substantive departure from Assumption~\ref{ass:finite}. Moreover, for
the present calculation the time-1 feasible set is a singleton, so the
recursion requires only the conditional means
\[
\E[C\mid A_0=a_0,Z_1=z_1,A_1=a_1].
\]

\begin{table}[ht]
\centering
\caption{Observational data from \citet{Naimi2017}, with the
  $(A_0,Z_1,A_1)=(0,1,0)$ cell count corrected to $60{,}657$ per the
  published web supplement (see text). $C$ = cell-mean CD4 count
  (cells/mm$^3$); $N=1{,}000{,}000$.}
\label{tab:hiv_obs}
\renewcommand{\arraystretch}{1.15}
\begin{tabular}{cccrr}
\toprule
$A_0$ & $Z_1$ & $A_1$ & $\E[C\mid A_0,Z_1,A_1]$ & $N$ \\
\midrule
0 & 0 & 0 &  87.29 & 209{,}271 \\
0 & 0 & 1 & 112.11 &  93{,}779 \\
0 & 1 & 0 & 119.65 &  60{,}657 \\
0 & 1 & 1 & 144.84 & 136{,}293 \\
1 & 0 & 0 & 105.28 & 134{,}781 \\
1 & 0 & 1 & 130.18 &  60{,}789 \\
1 & 1 & 0 & 137.72 &  93{,}903 \\
1 & 1 & 1 & 162.83 & 210{,}527 \\
\bottomrule
\end{tabular}
\end{table}

\noindent We now carry out in detail Theorem~\ref{thm:dp}'s backward induction for the \emph{always-treat} DTR:
the regime $d=(d_0,d_1)=(1,1)$ that sets $A_0=1$ and $A_1=1$.

\paragraph{Initialisation.} At $t=2$,
$\overline V_2^{(1,1)}(s)=\underline V_2^{(1,1)}(s)=Y(s)$ for every
$s\in\mathcal S_2$. This step is not carried out explicitly,  since we do not have the actual supported $C$ values.  Instead,  the
conditional means $\E[Y\mid A_0{=}a_0,Z_1{=}z_1,A_1{=}a_1]=\E[C\mid A_0{=}a_0,Z_1{=}z_1,A_1{=}a_1]$ are given directly (Table~\ref{tab:hiv_obs}). These conditional means will actually enter at the time-1 step of the backward induction algorithm. 

\paragraph{Time 1 step.} This step is unconfounded,  so $\cQ_1(s,a)$ is a
singleton at every cell,  and the upper and lower value functions
coincide:
\[
\overline V_1^{(1,1)}(1,z_1)=\underline V_1^{(1,1)}(1,z_1)=V_1^{(1,1)}(1,z_1)
=\E[C\mid A_0{=}1,Z_1{=}z_1,A_1{=}1],
\]
read directly off the $A_1{=}1$ rows of Table~\ref{tab:hiv_obs}:
\[
V_1^{(1,1)}(1,0)=130.18,
\qquad
V_1^{(1,1)}(1,1)=162.83.
\]

\paragraph{Time-0 step.} For a fixed first-stage action $a_0$, the local support relevant at the step-0 
transition cell is the two-element set $\mathcal{S}_1(1,a_0)=\{(a_0,0),(a_0,1)\}$.  Applying Section~\ref{subsec:default} at the cell
$(S_0,A_0)=(1,1)$, with local support $\mathcal S_1(1,1)=\{(1,0),(1,1)\}$ we get
\[
\cQ_0^{\mathrm{default}}(1,1)=\bigl\{q\in\Delta(\mathcal S_1(1,1)):
\underline\pi_0(z_1\mid 1,1)\leq q(1,z_1)\leq\bar\pi_0(z_1\mid 1,1),\
z_1\in\{0,1\}\bigr\},
\]
which is the Manski polytope supplied by the default construction of Section~\ref{subsec:default} (Step~4), since the cell $(S_0,A_0)=(1,1)$ has positive probability and carries the bidirected edge $A_0\leftrightarrow S_1$.  Here $q(1,z_1)$ denotes the mass $q$ assigns
to the state $s_1=(1,z_1)$. Since $q(1,0)+q(1,1)=1$, only $q(1,0)$ needs
bounding: the $z_1{=}1$ constraint is redundant given the $z_1{=}0$ one.  Then,  
\eqref{eq:manski_lower}--\eqref{eq:manski_upper} give
\[
\underline\pi_0(0\mid1,1)=P(Z_1{=}0,A_0{=}1),
\qquad
\bar\pi_0(0\mid1,1)=\underline\pi_0(0\mid1,1)+P(A_0\neq1),
\]
both computable directly from the cell counts of Table~\ref{tab:hiv_obs}: 
\[
\underline\pi_0(0\mid1,1)=\frac{134{,}781+60{,}789}{1{,}000{,}000}
=\frac{195{,}570}{1{,}000{,}000}\approx0.196,
\]
and, using $P(A_0 \neq 1)=P(A_0=0)=0.500$,
\[
\bar\pi_0(0\mid1,1)=\frac{195{,}570}{1{,}000{,}000}+0.5
\approx0.696.
\]
The always-treat step-0 bounds are
\[
\overline V_0^{(1,1)}(1)=\max_{q\,\in\,\cQ_0^{\mathrm{default}}(1,1)}\ \sum_{z_1\in\{0,1\}}q(1,z_1)\,V_1^{(1,1)}(1,z_1),
\]
\[
\underline V_0^{(1,1)}(1)=\min_{q\,\in\,\cQ_0^{\mathrm{default}}(1,1)}\ \sum_{z_1\in\{0,1\}}q(1,z_1)\,V_1^{(1,1)}(1,z_1).
\]
Since $\mathcal S_1(1,1)$ has only two points this optimisation is
analytic: the maximiser assigns as much mass as possible to the
higher-value state $z_1{=}1$, using the lower endpoint for $q(1,0)$,
while the minimiser assigns as much mass as possible to the lower-value
state $z_1{=}0$, using the upper endpoint for $q(1,0)$. Using the
unrounded endpoints determined by the displayed counts,
\[
\overline V_0^{(1,1)}(1)=\mathbf{156.445},
\qquad
\underline V_0^{(1,1)}(1)=\mathbf{140.120},
\]
with maximising kernel $(q(1,0),q(1,1))$ approximately $(0.196,0.804)$
and minimising kernel approximately $(0.696,0.304)$.

\paragraph{Root step.} Since $S_0\equiv1$ is fixed, the root step is trivial:
$\cQ_{-1}(*,\star)=\{P(S_0\in\cdot)\}$ is a point mass on $S_0=1$, therefore the $t=-1$ instance of Theorem~\ref{thm:dp} gives
\[
\bar\theta^{(1,1)}=\overline V_0^{(1,1)}(1)=\mathbf{156.445},
\qquad
\underline\theta^{(1,1)}=\underline V_0^{(1,1)}(1)=\mathbf{140.120}.
\]

\noindent Having evaluated always-treat in detail, we now apply the same
calculation to the other three static regimes. Together, the four
regimes considered here are
\[
\mathcal D_{\mathrm{stat}}
=
\{(0,0),(0,1),(1,0),(1,1)\},
\]
corresponding respectively to never-treat, treat-late, treat-early,
and always-treat. These four regimes do not exhaust the deterministic
DTR class: the time-1 action may depend on the observed history
\(S_1=(A_0,Z_1)\), as in a rule that treats at time~1 only when
\(Z_1=1\). We restrict attention here to the four static regimes,
since the purpose of this subsection is to illustrate fixed-regime
evaluation.

\noindent Treat-early reuses the same feasible set
\(\cQ_0^{\mathrm{default}}(1,1)\) as always-treat.  Never-treat and
treat-late instead have \(d_0=0\) and use
\(\cQ_0^{\mathrm{default}}(1,0)\), derived in the same way from the
\(A_0=0\) cells of Table~\ref{tab:hiv_obs}. Combining the appropriate feasible set with the corresponding time-1 step value
functions $V_1^d$,  the same maximisation used above for always-treat gives the
value functions and bounds for all four regimes, reported in
Tables~\ref{tab:hiv_v1} and~\ref{tab:hiv_summary}.\\


\begin{table}[ht]
\centering
\caption{Step-1 value functions, based on the reported cell means.}
\label{tab:hiv_v1}
\renewcommand{\arraystretch}{1.15}
\begin{tabular}{lcc}
\toprule
Regime & $V_1^d(d_0,z_1{=}0)$ & $V_1^d(d_0,z_1{=}1)$ \\
\midrule
Always-treat $(1,1)$ & 130.18 & 162.83 \\
Never-treat  $(0,0)$ &  87.29 & 119.65 \\
Treat-early  $(1,0)$ & 105.28 & 137.72 \\
Treat-late   $(0,1)$ & 112.11 & 144.84 \\
\bottomrule
\end{tabular}
\end{table}

\begin{table}[ht]
\centering
\caption{Causal bounds and plug-in values.
  All quantities in CD4 cells/mm$^3$, computed at full precision and
  rounded to three decimals for display.}
\label{tab:hiv_summary}
\renewcommand{\arraystretch}{1.18}
\begin{tabular}{lccc}
\toprule
Regime & $\underline\theta^d$ & Plug-in & $\bar\theta^d$ \\
\midrule
Always-treat $d=(1,1)$ & 140.120 & 150.059 & 156.445 \\
Never-treat  $(d=0,0)$ &  93.663 & 100.037 & 109.843 \\
Treat-early  $d=(1,0)$ & 115.156 & 125.031 & 131.376 \\
Treat-late   $(0,1)$ & 118.556 & 125.002 & 134.921 \\
\bottomrule
\end{tabular}
\end{table}

\noindent Table~\ref{tab:hiv_summary} also reports, for each regime, the plug-in
(g-formula) value obtained by selecting the observed kernel from the
corresponding Manski feasible set; since the observed kernel is itself a
feasible point, this value lies between $\underline\theta^d$ and
$\bar\theta^d$ in every row (for always-treat, $150.059\in[140.120,156.445]$,
using observed kernel $(0.391,0.609)$).

Our plug-in values agree with \citet{Naimi2017} to the reported precision
for never-treat, treat-early, and treat-late ($100.0$, $125.0$, $125.0$).
For always-treat we obtain $150.059$, which rounds to $150.1$ rather than the reported
$150.0$ with this small difference presumably the result of some
intermediate rounding.

The plug-in ranking favours always-treat ($150.059$), and this remains
true under the specific type of confounding structure that we allow here: 
Table~\ref{tab:hiv_summary} shows always-treat's
entire interval, $[140.120,156.445]$, lies strictly above every other
regime's interval,
\[
\underline{\theta}^{\,(1,1)}
=
140.120
>
134.921
=
\max_{d\in
\mathcal D_{\mathrm{stat}}\setminus\{(1,1)\}}
\overline{\theta}^{\,d}.
\]
Thus, among the four static regimes considered here, always-treat is
preferred both by the classical plug-in ranking and by interval
dominance under the confounding model.

All feasible sets used above are the unmodified default construction of
Section~\ref{subsec:default} for the maintained $P$, $G$, and local
structural successor family $\mathbf S$ specified above. Every cell
entering this calculation has positive observational probability:
$P(S_0=1)=1$, and every one of the eight $(A_0,Z_1,A_1)$ cell counts in
Table~\ref{tab:hiv_obs} is strictly positive. At time~0 the bidirected
edge $A_0\leftrightarrow S_1$ therefore gives the default Manski
feasible sets, while at time~1 the absence of
$A_1\leftrightarrow S_2$ identifies the transition kernel. The root law
is also identified. Consequently, Theorem~\ref{thm:sharp} implies that every regime-value
interval reported in Table~\ref{tab:hiv_summary} is sharp for the
maintained causal model class:
\[
\underline\theta^d
=
\inf_{M\in\cM(P,G,\mathbf S)}E_M^d[Y],
\qquad
\overline\theta^d
=
\sup_{M\in\cM(P,G,\mathbf S)}E_M^d[Y],
\]
with both endpoints attained, for every static regime $d$ considered
in this subsection. Sharpness holds regime by regime: the SCM attaining
the lower or upper endpoint for one regime need not attain the
corresponding endpoint for another.

Subtracting the separately sharp bounds for always-treat and never-treat
gives a valid outer interval for the contrast
$E^{(1,1)}[Y]-E^{(0,0)}[Y]$:
\[
[\underline\theta^{(1,1)}-\bar\theta^{(0,0)},\;
 \bar\theta^{(1,1)}-\underline\theta^{(0,0)}]
=[30.276,\;62.781],
\]
computed at full precision before rounding.  These bounds are not claimed to be
sharp for the contrast itself: its two endpoints come from separate
worst-case optimisations, one per regime,  and nothing here shows they
are jointly attainable under a single admissible SCM.

\subsection{Synthetic Biomarker Study: Reversal between Plug-in and Maximin
            Rankings}\label{subsec:biom}

Subsection~\ref{subsec:hiv} evaluated fixed regimes using a full-history
state and a binary successor component at the confounded transition. We
now study optimal regime selection under the maximin criterion of
Subsection~\ref{subsec:selection}. The main difference with the example in
Subsection~\ref{subsec:hiv} is that we now use
a three-state Markov model in which $S_t$ is a sufficient state summary
rather than the full observed history, while
retaining binary treatment. We also consider a nondegenerate baseline law
for the state $S_0$, in contrast with the fixed value of $S_0$ in Subsection~\ref{subsec:hiv}.

\subsubsection{Setup and Observational Law}\label{subsec:setup}

\paragraph{Variables and graph.}
State $S_t\in\{0,1,2\}$ (low/medium/high biomarker expression), treatment
$A_t\in\{0,1\}$ (standard/intensive), horizon $T=2$, outcome $Y=y(S_2)$
with $y(s)=1-s/2$. Thus higher terminal biomarker expression is adverse,
and maximizing $Y$ is equivalent to minimizing terminal biomarker
severity. Directed edges: $S_0\to A_0$, $S_0\to S_1$, $A_0\to S_1$,
$S_1\to A_1$, $S_1\to S_2$, $A_1\to S_2$. Bidirected edge:
$A_0\leftrightarrow S_1$ (contemporaneous confounding at step~0); no
bidirected edge $A_1\leftrightarrow S_2$.

\noindent We impose no structural restriction on the next state beyond the
three-state space itself. Thus, for every $s\in\{0,1,2\}$ and
$a\in\{0,1\}$, we take
\[
\mathcal S_1(s,a)=\mathcal S_2(s,a)=\{0,1,2\}.
\]
These sets constitute the local structural successor family
$\mathbf S$ for this example.

\paragraph{Observational law.}
The baseline state follows
\[
P(S_0=0)=0.2,\qquad P(S_0=1)=0.5,\qquad P(S_0=2)=0.3,
\]
so the root distribution is not degenerate, unlike Subsection~\ref{subsec:hiv}'s $S_0\equiv1$. Treatment propensities vary with observed baseline severity, so the fraction of patients observed under each action, and hence the default feasible set's coordinate width, varies across baseline states,
\[
\begin{aligned}
P(A_0=1\mid S_0=0)&=0.40,\\
P(A_0=1\mid S_0=1)&=0.70,\\
P(A_0=1\mid S_0=2)&=0.90,
\end{aligned}
\]
with $P(A_0=0\mid S_0=s_0)=1-P(A_0=1\mid S_0=s_0)$ in each case. 

\medskip
\noindent\textit{Step~0 transitions}: 

\begin{center}
\begin{tabular}{ccccc}
\toprule
$s_0$ & $a$ & $P(S_1=0\mid S_0{=}s_0,A_0{=}a)$ & $P(S_1=1\mid\cdot)$ & $P(S_1=2\mid\cdot)$ \\
\midrule
0 & 0 (standard)  & 0.55 & 0.30 & 0.15 \\
0 & 1 (intensive) & 0.30 & 0.45 & 0.25 \\
1 & 0 (standard)  & 0.75 & 0.20 & 0.05 \\
1 & 1 (intensive) & 0.50 & 0.40 & 0.10 \\
2 & 0 (standard)  & 0.05 & 0.15 & 0.80 \\
2 & 1 (intensive) & 0.15 & 0.35 & 0.50 \\
\bottomrule
\end{tabular}
\end{center}

\medskip
\noindent\textit{Step~1 transitions}:

\begin{center}
\begin{tabular}{cccccc}
\toprule
$s$ & $a$ & $P(S_2=0\mid s,a)$ & $P(S_2=1\mid\cdot)$ & $P(S_2=2\mid\cdot)$ & $\E[Y\mid s,a]$ \\
\midrule
0 & 0 & 0.60 & 0.30 & 0.10 & 0.750 \\
0 & 1 & 0.25 & 0.50 & 0.25 & 0.500 \\
1 & 0 & 0.30 & 0.40 & 0.30 & 0.500 \\
1 & 1 & 0.10 & 0.30 & 0.60 & 0.250 \\
2 & 0 & 0.05 & 0.15 & 0.80 & 0.125 \\
2 & 1 & 0.20 & 0.40 & 0.40 & 0.400 \\
\bottomrule
\end{tabular}
\end{center}

\noindent Unlike step~0, step~1 is unconfounded. We also assume
positivity, $P(A_1{=}a\mid S_1{=}s)>0$, for every $s\in\{0,1,2\}$ and
$a\in\{0,1\}$. Hence the post-intervention kernel coincides with the
observed conditional $P(S_2\mid S_1,A_1)$ (Section~\ref{sec:model}). The
numerical values of the step-1 treatment propensities do not enter the
calculations and are therefore not reported.

\subsubsection{Q-Learning Benchmark under Full Unconfoundedness}\label{subsec:qbench}

To ask what Q-learning would recommend under the counterfactual hypothesis
that step~0, like step~1, satisfied sequential exchangeability, we run the
ordinary backward recursion \eqref{eq:qlearning} of Corollary~\ref{cor:qlearning}
state by state, treating every observed conditional of
Section~\ref{subsec:setup} as if it were the interventional kernel. Under
this benchmark all relevant feasible sets are singletons.

\smallskip
\noindent \textbf{Initialisation} ($t=2$): $W_2(s)=y(s)$, i.e., $W_2(0)=1$, $W_2(1)=0.5$, $W_2(2)=0$.

\smallskip
\noindent\textbf{Recursion at $t=1$}: step~1 is genuinely unconfounded, so
$Q_1(s,a)$ is simply the observed conditional mean,
$Q_1(s,0)=(0.750,0.500,0.125)_{s=0,1,2}$ and
$Q_1(s,1)=(0.500,0.250,0.400)_{s=0,1,2}$. Maximising over $a$ at each $s$
gives
\[
d^*_1(s)=\begin{cases}0\quad(\text{standard})&s\in\{0,1\},\\1\quad(\text{intensive})&s=2,\end{cases}
\qquad
W_1(0)=0.750,\ \ W_1(1)=0.500,\ \ W_1(2)=0.400.
\]

\smallskip
\noindent\textbf{Recursion at $t=0$} (assuming sequential exchangeability): substituting each baseline
state's observed transition for the unidentified causal kernel, as in
the plug-in calculations of Subsection~\ref{subsec:hiv},
\eqref{eq:qlearning} gives the resulting action values
$Q_0^{\mathrm{plug}}(s_0,a)$:
\begin{align*}
Q_0^{\mathrm{plug}}(0,0)&=0.55(0.750)+0.30(0.500)+0.15(0.400)=0.6225, \\
Q_0^{\mathrm{plug}}(0,1)&=0.30(0.750)+0.45(0.500)+0.25(0.400)=0.55, \\
Q_0^{\mathrm{plug}}(1,0)&=0.75(0.750)+0.20(0.500)+0.05(0.400)=0.6825,\ \ \\
Q_0^{\mathrm{plug}}(1,1)&=0.50(0.750)+0.40(0.500)+0.10(0.400)=0.615, \\
Q_0^{\mathrm{plug}}(2,0)&=0.05(0.750)+0.15(0.500)+0.80(0.400)=0.4325,\ \ \\
Q_0^{\mathrm{plug}}(2,1)&=0.15(0.750)+0.35(0.500)+0.50(0.400)=0.4875,
\end{align*}
so the Q-learning-optimal first-stage greedy rule is itself adaptive,
\[
d_0^{\mathrm{plug}}(s_0)=\begin{cases}0&s_0\in\{0,1\},\\1&s_0=2,\end{cases}
\]
\[
W_0^{\mathrm{plug}}(0)=0.6225,\ \ W_0^{\mathrm{plug}}(1)=0.6825,\ \ W_0^{\mathrm{plug}}(2)=0.4875.
\]

\smallskip
\noindent\textbf{Recursion at $t=-1$}
Writing
$d^{\mathrm{plug}}:=(d_0^{\mathrm{plug}},d_1^*)$, the $t=-1$ instance
of Theorem~\ref{thm:dp} averages the state-specific optimal values over
$P(S_0)$, since
$ \cQ_{-1}^{\mathrm{default}}(*,\star) =\{P(S_0\in\cdot)\}$.
For a regime $d=(d_0,d_1^*)$, write
\[
V^{d}_{\mathrm{plug}}:=\sum_{s_0}P(s_0)\,Q_0^{\mathrm{plug}}\big(s_0,d_0(s_0)\big)
\]
for its root-averaged plug-in evaluation, obtained by substituting the
observed step-0 kernel for the true, here unidentified, one; by
Proposition~\ref{prop:basic}(c) it coincides with the true $V^d$
whenever all feasible sets are singletons, but need not do so under the
confounding maintained here. For $d^{\mathrm{plug}}$,
\[
\begin{aligned}
V^{d^{\mathrm{plug}}}_{\mathrm{plug}}
&=\sum_{s_0}P(s_0)\,W_0^{\mathrm{plug}}(s_0)\\
&=0.2(0.6225)+0.5(0.6825)+0.3(0.4875)=0.612.
\end{aligned}
\]

\subsubsection{Maximin Backward Induction under Step-0 Ambiguity}\label{subsec:maximinLP}

We now apply the general maximin recursion \eqref{eq:minimax_rec} of
Theorem~\ref{thm:optimal}, of which Section~\ref{subsec:qbench}'s
calculation was the singleton special case. Step~0 is actually
confounded, so here the minimisation in \eqref{eq:minimax_rec} is over a
genuine feasible set rather than a singleton.

\smallskip
\noindent\textbf{Initialisation} ($t=2$) and \textbf{recursion at $t=1$}
coincide exactly with Section~\ref{subsec:qbench}: step~1 is unconfounded
regardless of the step-0 ambiguity, so $\cQ_1(s,a)=\{P(S_2\mid
S_1{=}s,A_1{=}a)\}$ is a singleton at every cell, the minimisation in
\eqref{eq:minimax_rec} has nothing to minimise over, and the same $d_1^*$
and $W_1$ obtained above apply unchanged.

\smallskip
\noindent\textbf{Recursion at $t=0$}: Since we assume that this step is confounded
here $\cQ_0(s_0,a)$ is \emph{not} a
singleton, and this is the first point in the example where a genuine
feasible set is required. The default construction of
Section~\ref{subsec:default} assigns each confounded cell the Manski
polytope with lower and upper endpoints
\[
\underline\pi_0(s'\mid s_0,a)=P(A_0{=}a\mid S_0{=}s_0)\,P(S_1{=}s'\mid S_0{=}s_0,A_0{=}a),
\]
\[
\overline\pi_0(s'\mid s_0,a)=\underline\pi_0(s'\mid s_0,a)+P(A_0\neq a\mid S_0{=}s_0).
\]
This construction is applied at each of the six $(s_0,a)$ cells using the
propensities and transition rows of Section~\ref{subsec:setup}'s table.
Table~\ref{tab:manski_all} collects
the resulting six polytopes.

\begin{table}[h]
\centering
\begin{tabular}{cccccc}
\toprule
$s_0$ & $a$ & $[\underline\pi,\overline\pi]$ at $s'=0$ & at $s'=1$ & at $s'=2$ & \shortstack{coordinate width\\$P(A_0\neq a\mid s_0)$}\\
\midrule
0 & 0 & $[0.330,\,0.730]$ & $[0.180,\,0.580]$ & $[0.090,\,0.490]$ & 0.40 \\
0 & 1 & $[0.120,\,0.720]$ & $[0.180,\,0.780]$ & $[0.100,\,0.700]$ & 0.60 \\
1 & 0 & $[0.225,\,0.925]$ & $[0.060,\,0.760]$ & $[0.015,\,0.715]$ & 0.70 \\
1 & 1 & $[0.350,\,0.650]$ & $[0.280,\,0.580]$ & $[0.070,\,0.370]$ & 0.30 \\
2 & 0 & $[0.005,\,0.905]$ & $[0.015,\,0.915]$ & $[0.080,\,0.980]$ & 0.90 \\
2 & 1 & $[0.135,\,0.235]$ & $[0.315,\,0.415]$ & $[0.450,\,0.550]$ & 0.10 \\
\bottomrule
\end{tabular}
\caption{Manski feasible sets $\cQ_0^{\mathrm{default}}(s_0,a)$ at every
baseline state and action, obtained from
$\underline\pi_0(s'\mid s_0,a)=P(A_0{=}a\mid s_0)\,P(S_1{=}s'\mid s_0,a)$,
with common coordinate width $P(A_0\neq a\mid s_0)$.}
\label{tab:manski_all}
\end{table}

\noindent The maximin backward step \eqref{eq:minimax_rec} evaluates, at every
baseline state $s_0$ and action $a$,
\[
L(s_0,a)=\min_{q\in\cQ_0^{\mathrm{default}}(s_0,a)}\big\{0.750\,q_0+0.500\,q_1+0.400\,q_2\big\},
\]
using the continuation values $W_1(0)=0.750$, $W_1(1)=0.500$,
$W_1(2)=0.400$ obtained in the previous step.  For instance,  at
$s_0=1,a=1$ this is the linear programme
\[
\begin{aligned}
&\min_{q_0,q_1,q_2}\ 0.750\,q_0+0.500\,q_1+0.400\,q_2\\
&\text{s.t.}\quad q_0+q_1+q_2=1,\\
&\phantom{\text{s.t.}\quad} 0.350\le q_0\le0.650,\quad 0.280\le q_1\le0.580,\quad 0.070\le q_2\le0.370,
\end{aligned}
\]
a genuine linear program over the two-dimensional polytope cut out by
these box constraints and the simplex, in contrast to the single free
coordinate of the binary example of Section~\ref{subsec:hiv}. Since the
objective coefficients decrease in $s'$, the minimiser shifts as much mass
as the box constraints allow onto the lowest-coefficient state ($s'=2$) and as
little as possible onto the highest ($s'=0$), giving $q^*=(0.350,0.280,0.370)$ and
$L(1,1)=0.5505$. 

\noindent Solving the analogous linear program by the same method at each of the
remaining five state--action cells gives the worst-case values collected
in Table~\ref{tab:maximin_all}, together with the resulting maximin-optimal
action at each baseline state, $d_0^{\mathrm{mm}}(s_0)\in\arg\max_a
L(s_0,a)$, and the corresponding value $W_0^{\mathrm{mm}}(s_0)=\max_a
L(s_0,a)$.

\begin{table}[h]
\centering
\begin{tabular}{ccccc}
\toprule
$s_0$ & $L(s_0,0)$ & $L(s_0,1)$ & $d_0^{\mathrm{mm}}(s_0)$ & $W_0^{\mathrm{mm}}(s_0)$ \\
\midrule
0 & 0.5335   & 0.46     & 0 (standard)  & 0.5335   \\
1 & 0.48475  & 0.5505   & 1 (intensive) & 0.5505   \\
2 & 0.40325  & 0.47875  & 1 (intensive) & 0.47875  \\
\bottomrule
\end{tabular}
\caption{Worst-case action values, the maximin-optimal first-stage
action, and the resulting value $W_0^{\mathrm{mm}}$ at each baseline
state.}
\label{tab:maximin_all}
\end{table}

\noindent Here the maximin-optimal first-stage action is itself adaptive,
like Section~\ref{subsec:qbench}'s plug-in rule: the two benchmarks agree at
$s_0=0$ (both prefer standard) and at $s_0=2$ (both prefer intensive), but
disagree at $s_0=1$, where plug-in prefers standard while maximin instead
prefers intensive.

\smallskip
\noindent\textbf{Root step}: writing
$d^{\mathrm{mm}}:=(d_0^{\mathrm{mm}},d_1^*)$ for the resulting DTR, the
root step performs the aggregation
\[
\begin{aligned}
\underline\theta^{d^{\mathrm{mm}}}
&=\sum_{s_0}P(s_0)\,W_0^{\mathrm{mm}}(s_0)\\
&=0.2(0.5335)+0.5(0.5505)+0.3(0.47875)=0.525575.
\end{aligned}
\]
For comparison, the same root averaging applied to $d^{\mathrm{plug}}$'s
adaptive first-stage rule of Section~\ref{subsec:qbench} gives
\[
\underline\theta^{d^{\mathrm{plug}}}
=\sum_{s_0}P(s_0)\,L\big(s_0,d_0^{\mathrm{plug}}(s_0)\big)
=0.2(0.5335)+0.5(0.48475)+0.3(0.47875)=0.4927,
\]
confirming $\underline\theta^{d^{\mathrm{mm}}}>\underline\theta^{d^{\mathrm{plug}}}$
as required by the maximin optimality of $d^{\mathrm{mm}}$ among all DTRs.

\subsubsection{Reversal and Sharpness}\label{subsec:reversal}

All feasible sets above are the unmodified default construction of
Section~\ref{subsec:default} for the maintained $P$, $G$, and local
structural successor family $\mathbf S$ specified in
Subsection~\ref{subsec:setup}. At step~0,
$A_0\leftrightarrow S_1$ is present, so the positive-probability cells
give the default Manski feasible sets on
$\mathcal S_1(s_0,a)=\{0,1,2\}$. At step~1, the absence of
$A_1\leftrightarrow S_2$ and positivity identify the transition
kernels, while the root law $P(S_0)$ is identified.
Theorem~\ref{thm:sharp} therefore implies that the regime-value bounds
computed from these default feasible sets are sharp for the maintained
causal model class, not merely valid outer bounds.


The regimes $d^{\mathrm{plug}}=(d_0^{\mathrm{plug}},d_1^*)$ and
$d^{\mathrm{mm}}=(d_0^{\mathrm{mm}},d_1^*)$ are both adaptive at each
stage and differ only at $s_0=1$, where $d^{\mathrm{plug}}$ retains
standard treatment and $d^{\mathrm{mm}}$ switches to intensive.
Table~\ref{tab:reversal} reports each regime's root-averaged plug-in
value and worst-case bound $\underline\theta^d$.

\begin{table}[h]
\centering
\begin{tabular}{lccc}
\toprule
Regime & $d_0(s_0)$ & Plug-in value & Worst-case value $\underline\theta^d$ \\
\midrule
$d^{\mathrm{plug}}$ & standard if $s_0\in\{0,1\}$, else intensive & 0.612 & 0.493 \\
$d^{\mathrm{mm}}$   & standard if $s_0{=}0$, else intensive & 0.578 & 0.526 \\
\bottomrule
\end{tabular}
\caption{Reversal between the Q-learning-optimal and maximin-optimal DTRs, aggregated over $P(S_0)$.}
\label{tab:reversal}
\end{table}

$d^{\mathrm{plug}}$ and $d^{\mathrm{mm}}$ are each optimal, but for
different criteria resting on different assumptions. $d^{\mathrm{plug}}$
is Q-learning-optimal under the counterfactual assumption that step~0 is
identified: it substitutes the observed kernel for the true one at every
state and ranks regimes by comparing these single point values.
$d^{\mathrm{mm}}$ makes no such assumption: with step~0 genuinely
confounded, each regime's value is known only up to an interval
$[\underline\theta^d,\bar\theta^d]$, and maximin ranks regimes by the
lower endpoint: the value guaranteed under the worst kernel compatible with the observational law. At $s_0=1$ standard treatment has the higher point
value but the lower worst-case bound, so the two criteria disagree;
since $P(S_0=1)=0.5$, this one state drives the aggregate reversal in
Table~\ref{tab:reversal}. Both DTRs remain globally optimal for their
own criterion, and the reversal is not a contradiction: $d^{\mathrm{plug}}$'s
plug-in value $0.612$ still lies inside its own valid interval
$[0.493,0.654]$, since the observed kernel is itself a member of its
Manski feasible set at every $s_0$.

The reversal traces to an asymmetry in the Manski sets' coordinatewise
slack, which equals the unobserved-action fraction $P(A_0\neq a\mid
S_0=s_0)$. For standard treatment, the observed-action fraction
decreases with baseline severity,
$P(A_0=0\mid S_0=s_0)=(0.60,0.30,0.10)$ at $s_0=0,1,2$, so its
coordinatewise slack correspondingly increases,
$P(A_0\neq0\mid S_0=s_0)=(0.40,0.70,0.90)$. At $s_0=0,2$ the
plug-in-preferred action is also the well-supported one, so no tension
arises. At $s_0=1$, standard is the rarer, more exposed action
(coordinatewise slack $0.70$), and its worst case falls below the
better-supported intensive alternative, flipping the ranking. Maximin
selects the higher lower bound state by state; plug-in ignores this
exposure entirely.

\section{Conclusion}

We have developed a framework for evaluating and selecting dynamic
treatment regimes in the presence of contemporaneous unmeasured
confounding. At each state--action pair, the observational law, the causal
graph, and the structural successor sets determine a set of possible
transition probabilities. These sets are then combined by backward
induction. For the default sets considered in this paper,
Theorem~\ref{thm:sharp} shows that the resulting lower and upper bounds are
exactly the smallest and largest expected outcomes allowed by the assumed
causal model. This sharpness depends on the one-step structure of the
confounding; the backward-induction calculation itself applies to any
rectangular family of feasible sets (Remark~\ref{rem:two_claims}).

The same method can be used with a maximin criterion to select a treatment
regime. When all local transition probabilities are identified, evaluation
reduces to the classical g-formula of \citet{Robins1986}, while selection
reduces to Q-learning of \citet{Murphy2003}. In this sense, the results
extend these classical procedures to the contemporaneous-confounding
setting studied here.

The biomarker example in Section~5.2 shows how the choice of treatment
regime can change once the confounding is taken into account. The plug-in
calculation treats the observed transition probabilities as if they were
identified, while the maximin calculation allows all transition
probabilities that are compatible with the observational law and the causal assumptions.
The two procedures disagree at the baseline state where the treatment
favoured by the plug-in calculation is also the one 
that is less likely to be assigned.
Its worst-case expected outcome is therefore lower, and the
maximin rule selects the alternative treatment. Thus a ranking based on
point identification can differ from the ranking obtained when the
remaining causal uncertainty is taken seriously.

The analysis is restricted to confounding that acts within a single stage.
Assumption~\ref{ass:npc} rules out latent variables that persist across
several periods. Allowing such persistent confounding would require a
different factorization and would likely lead to ambiguity sets that are
not rectangular. We expect that some validity results may survive in that
more general setting, but sharpness should not be expected in general.

Throughout the paper we treat the observational law as known. With finite data, sampling uncertainty would enter in addition to causal identification uncertainty. Developing inference for the resulting partially identified bounds is a separate problem, although the local feasible-set and backward-induction formulation developed here provides a natural starting point for such an extension. Another question is when the regime selected
by the maximin criterion is the same as the regime that would be selected
under point identification. Finally, the financial interpretation in
Remark~\ref{rem:tree_methods} suggests that other ideas from pricing under
uncertainty may also be useful in this setting.

If the analyst specifies where unmeasured confounding may occur and which
next states are structurally possible, the method does not require an
additional sensitivity parameter, an instrument, or a separately chosen
ambiguity set. Under the assumptions of this paper, these ingredients are
enough to obtain sharp bounds and to compare treatment regimes by their
worst-case expected outcomes. The calculations remain local: the global
problem is solved through a sequence of smaller problems, one at each
state--action cell. When there is no confounding, the same construction
reduces to the standard identified methods.

\appendix

\section{Proof of Theorem~\ref{thm:sharp}}\label{app:thm_sharp}

The reverse inclusion in Theorem~\ref{thm:sharp} uses the following
statewise construction.

\begin{lemma}[Local realizability]\label{lem:local_real}
Fix $t$, $s\in\cS_t$, write $a^*=d_t(s)$, and let
$q^s\in\cQ_t^{\mathrm{default}}(s,a^*)$. 
There exist a local treatment mechanism and
a local transition mechanism, with the transition mechanism defined for every
\(a \in \mathcal A\), such that:
\begin{enumerate}[label=(\roman*)]
\item if $P(S_t=s)>0$, they reproduce
$P(A_t,S_{t+1}\mid S_t=s)$;
\item $\kappa_{M,t}(\cdot\mid s,a^*)=q^s$;
\item if $A_t\leftrightarrow S_{t+1}$ is absent from $G$, the treatment
and transition mechanisms share no latent exogenous parent;
\item for every $a\in\cA$, the transition mechanism takes values in
$\cS_{t+1}(s,a)$.
\end{enumerate}
\end{lemma}

\begin{proof}
Since $q^s\in\Delta(\cS_{t+1}(s,a^*))$, we have
$\supp q^s\subseteq\cS_{t+1}(s,a^*)$. Also, whenever
$P(S_t=s,A_t=a)>0$, Remark~\ref{rem:observed_support_compatibility}
gives
$\supp P(S_{t+1}\mid S_t=s,A_t=a)\subseteq\cS_{t+1}(s,a)$.

First suppose $P(S_t=s)>0$. Write
$\alpha(a)=P(A_t=a\mid S_t=s)$ and, whenever $\alpha(a)>0$, let
$g_a=P(S_{t+1}\mid S_t=s,A_t=a)$. Thus
$\supp g_a\subseteq\cS_{t+1}(s,a)$ for every observed action.

If $\alpha(a^*)=0$, generate $A_t\sim\alpha$ from private treatment
noise. Using independent transition noise, choose for every action
$a$ with $\alpha(a)>0$ a map into $\cS_{t+1}(s,a)$ having law $g_a$,
and under $a^*$ a map into $\cS_{t+1}(s,a^*)$ having law $q^s$.
For any remaining null action choose an arbitrary constant in its
structural successor set. This defines the transition mechanism for
\emph{all} $a\in\cA$, reproduces the observational law for every
observed action, and gives structural kernel $q^s$ under $a^*$.
Because the treatment and transition noises are independent and
disjoint, condition~(iii) also holds whenever the bidirected edge is
absent.

If $\alpha(a^*)>0$ and $A_t\leftrightarrow S_{t+1}$ is absent from $G$,
the default set is the singleton $\{g_{a^*}\}$, so $q^s=g_{a^*}$.
Generate $A_t\sim\alpha$ from private treatment noise and, independently,
generate the transition under each observed action $a$ with law $g_a$;
at null actions use arbitrary constants in the corresponding
$\cS_{t+1}(s,a)$. Again this specifies one transition mechanism over
the entire action space. Conditions~(i)--(iv) follow, and in particular
the two structural equations use disjoint exogenous sources.

Finally suppose $\alpha(a^*)>0$ and
$A_t\leftrightarrow S_{t+1}$ is present. Put
$\alpha^*=\alpha(a^*)$ and
$\pi(z)=P(A_t=a^*,S_{t+1}=z\mid S_t=s)
       =\alpha^*g_{a^*}(z)$.
If $\alpha^*=1$, the Manski inequalities force $q^s=g_{a^*}$, and the
preceding independent-noise construction works, with treatment
identically equal to $a^*$.

Suppose therefore that $0<\alpha^*<1$, and define
\[
r(z)=\frac{q^s(z)-\pi(z)}{1-\alpha^*}.
\]
The lower Manski inequalities imply $r(z)\geq0$, while
$\sum_z r(z)=1$ because $\sum_zq^s(z)=1$ and
$\sum_z\pi(z)=\alpha^*$. Hence $r$ is a probability distribution.
Moreover, $\supp r\subseteq\cS_{t+1}(s,a^*)$, since both $q^s$ and
$\pi$ vanish outside that set.

Let $B$ have law $\alpha$ and construct one shared latent variable
$U=(B,(Z_a)_{a\in\cA})$. For every action $a$ with $\alpha(a)>0$,
choose $\mathcal L(Z_a\mid B=a)=g_a$, and additionally choose
$\mathcal L(Z_{a^*}\mid B\neq a^*)=r$. All remaining coordinates
$Z_b$ are chosen arbitrarily inside $\cS_{t+1}(s,b)$. Thus a response
$Z_a$ is specified for every $a\in\cA$, not only for $a^*$.
Set $A_t=B$ and let the transition under action $a$ be $S_{t+1}=Z_a$.
Then, for every observed action,
$P_M(A_t=a,S_{t+1}=z\mid S_t=s)=\alpha(a)g_a(z)$, while under
the intervention \(A_t\leftarrow a^*\), the treatment equation
\(A_t=B\) is overridden, but the latent variable \(B\) retains its
original distribution. The resulting next state is \(Z_{a^*}\).
Using the law of total probability over the partition
\(\{B=a^*\},\{B\neq a^*\}\), and the specified
conditional laws of \(Z_{a^*}\), we obtain
\[
\begin{aligned}
P_M(Z_{a^*}=z)
&=P_M(B=a^*,Z_{a^*}=z)
  +P_M(B\neq a^*,Z_{a^*}=z)\\
&=\alpha^*P_M(Z_{a^*}=z\mid B=a^*)
 +(1-\alpha^*)P_M(Z_{a^*}=z\mid B\neq a^*)\\
&=\alpha^*g_{a^*}(z)+(1-\alpha^*)r(z)\\
&=\pi(z)+(1-\alpha^*)r(z)\\
&=q^s(z),
\end{aligned}
\]
Hence
\(\kappa_{M,t}(\cdot\mid s,a^*)=q^s\), as required. The shared latent is permitted by the bidirected edge, and every transition response lies in its corresponding structural successor set.

If instead $P(S_t=s)=0$, there is no observational law to preserve.
Choose the response under $a^*$ to have law $q^s$, and choose arbitrary
responses inside $\cS_{t+1}(s,a)$ for every other action, using private
treatment and transition noises. Thus the mechanism is again defined
for all $a\in\cA$; condition~(i) is vacuous, while (ii)--(iv) hold.
When the bidirected edge is absent, the two noises are chosen
independently, giving~(iii).
\end{proof}

\begin{proof}[Proof of Theorem~\ref{thm:sharp}]
We prove the two inclusions.

\medskip
\noindent\emph{Forward inclusion.}
Fix $M\in\cM(P,G,\mathbf S)$. We show that every structural kernel of
$M$ belongs to the corresponding default feasible set.

Consider first a positive confounded cell $(t,s,a)$ and set
$Z_a=f^S_{M,t+1}(s,a,U_t,\varepsilon_{t+1}^S)$ and
$E_a=\{A_t=a\}$. The state $S_t$ is determined recursively by exogenous
variables from stages strictly before $t$, whereas $Z_a$, with $s$ and
$a$ fixed, depends only on the fresh variables
$(U_t,\varepsilon_{t+1}^S)$. Assumption~\ref{ass:npc} therefore gives
$Z_a\perp S_t$. Hence
\[
\kappa_{M,t}(z\mid s,a)
=
P_M(E_a,Z_a=z\mid S_t=s)
+
P_M(E_a^c,Z_a=z\mid S_t=s).
\]
On $E_a$, consistency gives $Z_a=S_{t+1}$. Since $P_M=P$, the first
term is $\underline\pi_t(z\mid s,a)$, while the second lies between
$0$ and $P(A_t\neq a\mid S_t=s)$. Thus
$\underline\pi_t(z\mid s,a)\leq\kappa_{M,t}(z\mid s,a)
\leq\bar\pi_t(z\mid s,a)$. Since $M$ respects $\mathbf S$,
$\supp\kappa_{M,t}(\cdot\mid s,a)\subseteq\cS_{t+1}(s,a)$, so the
kernel belongs to the default Manski set.

At a positive unconfounded cell, the structural kernel equals the
observed conditional and hence belongs to the singleton default set.
At a null cell, the default set is
$\Delta(\cS_{t+1}(s,a))$, which contains the structural kernel because
$M$ respects the local structural successor sets. The root kernel
belongs to the singleton root set. Thus this membership holds at
\emph{every} cell $(t,s,a)$, not only at cells selected by the regime.

Theorem~\ref{thm:factorization} now gives
$P_M^d\in\cP^d(\cQ^{\mathrm{default}})$. Since $M$ was arbitrary,
\[
\cI^d(P,G,\mathbf S)
\subseteq
\cP^d(\cQ^{\mathrm{default}}).
\]

\medskip
\noindent\emph{Reverse inclusion.}
Take $Q\in\cP^d(\cQ^{\mathrm{default}})$ and write
$q_{t,s}=q_t(\cdot\mid s,d_t(s))$ for its selected local kernels.
For every $(t,s)$, apply Lemma~\ref{lem:local_real} with
$a^*=d_t(s)$ and $q^s=q_{t,s}$. Importantly, although only the kernel
under $a^*$ is prescribed by $Q$, the lemma constructs a complete local
treatment and transition mechanism for every action $a\in\cA$, as
required for membership in $\cM(P,G,\mathbf S)$.

We paste these statewise constructions into a single SCM. For each
stage $t$, collect the local shared variables and private noises as
\[
U_t=(U_{t,s})_{s\in\cS_t},\qquad
\varepsilon_t^A=(\varepsilon_t^{A,s})_{s\in\cS_t},\qquad
\varepsilon_{t+1}^S=(\varepsilon_{t+1}^{S,s})_{s\in\cS_t},
\]
taking the state-indexed components mutually independent except for the
dependence already internal to each shared $U_{t,s}$. The global
structural maps simply select the component indexed by the realized
state $s$. When $A_t\leftrightarrow S_{t+1}$ is absent, $U_t$ is taken
degenerate in the transition equation and the treatment and transition
maps use the separate private-noise vectors supplied by the lemma.
Choose these stagewise exogenous vectors independently across $t$ and
independently of the exogenous variable generating $S_0\sim P(S_0)$.
Hence Assumption~\ref{ass:npc} holds, the resulting SCM is compatible
with $G$, and Lemma~\ref{lem:local_real}(iv) ensures that it respects
$\mathbf S$.

It remains to verify that \(P_M=P\). Let
\[
H_t=(S_0,A_0,\ldots,A_{t-1},S_t).
\]
We prove by induction that \(P_M(H_t=h_t)=P(H_t=h_t)\) for every
prefix \(h_t\). This holds at \(t=0\) by construction. Suppose it holds
at time \(t\), and let \(s\) be the terminal state of \(h_t\). If
\(P(H_t=h_t)=0\), then every extension of \(h_t\) has probability zero
under both laws. Otherwise \(P(S_t=s)>0\), and Proposition~\ref{prop:obs_factorization},
Lemma~\ref{lem:local_real}, and the induction hypothesis give
\[
\begin{aligned}
P_M(H_{t+1}=(h_t,a,s'))
&=P_M(H_t=h_t)P_M(A_t=a,S_{t+1}=s'\mid S_t=s)\\
&=P(H_t=h_t)P(A_t=a,S_{t+1}=s'\mid S_t=s)\\
&=P(H_{t+1}=(h_t,a,s')).
\end{aligned}
\]
Thus the induction holds through \(t=T\), giving equality of the
state--action trajectory laws; since \(Y=y(S_T)\), \(P_M=P\).


Thus $M\in\cM(P,G,\mathbf S)$. Finally,
Lemma~\ref{lem:local_real}(ii) gives
$\kappa_{M,t}(\cdot\mid s,d_t(s))=q_{t,s}$ at every $(t,s)$.
Theorem~\ref{thm:factorization} then shows that the interventional
factorization of $P_M^d$ is exactly the factorization defining $Q$.
Hence $P_M^d=Q$, and therefore
\[
\cP^d(\cQ^{\mathrm{default}})
\subseteq
\cI^d(P,G,\mathbf S).
\]
Combining the two inclusions proves the result.
\end{proof}

\bibliographystyle{plainnat}

\end{document}